\documentclass{tran-l}

\usepackage[T1]{fontenc}
\usepackage[latin9]{inputenc}
\usepackage{color}
\usepackage{array}
\usepackage{float}
\usepackage{mathrsfs}
\usepackage{mathtools}
\usepackage{amsthm}
\usepackage{amstext}
\usepackage{amssymb}
\usepackage{stmaryrd}
\usepackage{graphicx}
\usepackage{pgfplots}

\usepackage{mathcomp} 
\usepackage{tikz-cd}

\makeatletter

\theoremstyle{plain}
\newtheorem{thm}{\protect\theoremname}

\theoremstyle{plain}

\theoremstyle{plain}

\theoremstyle{plain}
\newtheorem{lem}[thm]{\protect\lemmaname}
\theoremstyle{definition}

\theoremstyle{definition}
\newtheorem{defn}[thm]{\protect\definitionname}
\theoremstyle{plain}
\newtheorem{notation}[thm]{Notation}
\theoremstyle{plain}

\newtheorem{thmletter}{Theorem}

\AtBeginDocument{
	
}

\newcommand{\Rmnum}[1]{\expandafter\@slowromancap\romannumeral #1@}

\makeatother

  \providecommand{\corollaryname}{Corollary}
  \providecommand{\examplename}{Example}
  \providecommand{\lemmaname}{Lemma}
  \providecommand{\propositionname}{Proposition}
  \providecommand{\theoremname}{Theorem}
  \providecommand{\definitionname}{Definition}

\newcommand{\Id}{\operatorname{Id}}
\newcommand{\GL}{\operatorname{GL}}

\newcommand{\PGL}{\operatorname{PGL}}
\newcommand{\Aut}{\operatorname{Aut}}
\newcommand{\Der}{\operatorname{Der}}
\newcommand{\E}{e}
\newcommand{\F}{\mathbf{f}}
\newcommand{\w}{\mathbf{w}}

\newcommand{\Grp}{\operatorname{Grp}}
\newcommand{\AcGrp}{\operatorname{Grp}^{\operatorname{ac}}}

\newcommand{\weight}{\operatorname{weight}}

\begin{document}

\title{groupoids derived from the simple elliptic singularities}

\author{Chuangqiang~Hu}
 \thanks{Beijing Institute of Mathematical Sciences and Applications,  Beijing 101408, P. R.  China\\
 E-mail: huchq@bimsa.cn}

\author{Stephen S.-T.~Yau}
 \thanks{Corresponding author: Stephen~S.-T. Yau\\
 Beijing Institute of Mathematical Sciences and Applications,  Beijing 101408, P. R.  China\\
  E-mail: yau@uic.edu}
\author{Huaiqing~Zuo}
 \thanks{Department of Mathematical Sciences, Tsinghua University, Beijing 100084, P. R.  China\\
 E-mail: hqzuo@mail.tsinghua.edu.cn}
 
  \thanks{Zuo was supported by NSFC Grant 12271280. Yau was supported by Tsinghua University start-up fund and Tsinghua University Education Foundation fund (042202008).}
  
  \thanks{\textbf{Data availability}: Data sharing not applicable to this article as no datasets were generated or analysed during the current study.}

\maketitle

\begin{abstract}
	K. Saito's classification of simple elliptic singularities includes three families of weighted homogeneous singularities: $ \tilde{E}_{6}, \tilde{E}_7$, and $ \tilde{E}_8 $. For each family, the isomorphism classes can be distinguished by K.~Saito's $j$-functions. By applying the Mather-Yau theorem, which states that the isomorphism class of an isolated hypersurface singularity is completely determined by its $k$-th moduli algebra, M. Eastwood demonstrated explicitly that one can directly recover K.~Saito's $j$-functions from the zeroth moduli algebras. This research aims to generalize M. Eastwood's result through meticulous computation of the groupoids associated with simple elliptic singularities. We not only directly retrieve K.~Saito's $j$-functions from the $k$-th moduli algebras but also elucidate the automorphism structure within the $k$-th moduli algebras. We derive the automorphisms using the methodology of the $k$-th Yau algebra and establish a Torelli-type theorem for the  $\tilde{E}_7 $-family when $k=1$. In contrast, we find that the Torelli-type theorem is inapplicable for the first Yau algebra in the $ \tilde{E}_6 $-family. 
	By considering the first Yau algebra as a module rather than solely as a Lie algebra, we can impose constraints on the coefficients of the transformation matrices, which facilitates a straightforward identification of all isomorphisms. Our new approach also provides a simple verification of the result by Chen, Seeley, and Yau concerning the zeroth moduli algebras.
	
 MSC(2020):  Primary14B05; Secondary 32S05
  
  Keywords: simple elliptic singularity, Tjurina algebra, Yau algebra, weighted homogeneous singularity, Groupoid
\end{abstract}


\section{Introduction}\label{sec:introduction}
 
     We present necessary definitions and auxiliary results concerning the moduli algebras of singularities.
     Let $ \mathcal{O}_n:= \mathbb{C} [[x_1, \ldots ,x_n ]] $ be a formal power series ring over $ \mathbb{C} $ with maximal ideal $  m $.
     Let $ \mathcal{V}_f $ be a germ of an isolated hypersurface singularity at the origin in $ \mathbb{C}^n $ represented as the zero locus of polynomial (or analytic function) $f = f(x_1, \ldots, x_n) $.
	\begin{defn}
		The $ k $-th moduli algebra of $  \mathcal{V}_f  $ is defined by 
     \[ \mathcal{A}^{k}( \mathcal{V}_f ) = \mathcal{O}_n /\langle  f, m^k J(f) \rangle , \]
     where $ J(f) = \langle \partial_{1}(f), \ldots, \partial_{ n}(f)\rangle $ denotes the Jacobi ideal of $f $.
	\end{defn}
      It is convenient to denote the local function algebra (or coordinate ring) of $ \mathcal{V}_f $ by 
     \[ \mathcal{A}^{\infty}(\mathcal{V}_f ) = \mathcal{O}_n  /\langle  f \rangle . \]
     The zeroth moduli algebra is also called Tjurina algebra, and its dimension is referred to as the Tjurina number. It is well known that the moduli algebra
     can serve as a base space of versal deformation of singularities \cite{Arnold2012Singularities}. In our recent work \cite{hu2024k}, we computed the dimension of the $k$-th Tjurina algebra of weighted homogeneous singularities. For  further related studies on on Tjurina algebra, we refer to \cite{Greuel2007, trang1976invariance,milnor1970isolated, hussain2023k}.

     The $k$-th moduli algebras are important due to the Mather-Yau theorem \cite{Mather1982Classification}. The extension of Mather-Yau theorem to positive characteristic was studied in \cite{GreuelPham2017}. We restate the theorem in a slightly different version \cite{Greuel2007}.

	 \begin{thm}[Mather-Yau theorem]
		Let $ f_i $ for $i = 0,1$ denote analytic functions associated with isolated singularities $ \mathcal{V}_{f_i}$ respectively. Then the following conditions are equivalent:
		\begin{enumerate}
			\item The function $f_1$ is contact equivalent to $ f_2 $, i.e., there exists some automorphism $ \phi \in \Aut(\mathcal{O}_n) $ and a unit $ u \in \mathcal{O}_n^* $ such that 
			\[
				\phi(f_1) = u \cdot f_2.  
			\]  
			\item For all $ k\geqslant 0 $, there exists some isomorphism $ \mathcal{A}^k(\mathcal{V}_{f_1}) \cong \mathcal{A}^k(\mathcal{V}_{f_2}) $ of $\mathbb{C}$-algebras. 
			\item There exists some $ k \geqslant 0 $ such that $ \mathcal{A}^k(\mathcal{V}_{f_1}) \cong \mathcal{A}^k(\mathcal{V}_{f_2}) $ as $\mathbb{C}$-algebras. 
		\end{enumerate} 
	 \end{thm}
	 Among all isolated hypersurface singularities, weighted homogeneous singularities have been of particular interest.
	 Recall that a polynomial $ f( x_1 , \cdots, x_n )$ is weighted homogeneous of type	 $(w_1, w_2 , \cdots, w_n )$, where $ w_1 , w_2 , \cdots, w_n $ are fixed positive rational numbers, if it can be expressed as a linear combination of monomials $x_1^{i_1} x_2^{i_2} \cdots x_n^{i_n} $
	 for which 
	 \[ i_1 w_1 + i_2 w_2 + \cdots + i_n w_n = W 
	 \] for some constant $W$. 
	 
	 We view $\mathcal{O}_n$ as a graded algebra by assigning the weight of $x_i$ to be $ w_i $. Let $ I_1, I_2 $ be two homogeneous ideals. For the local algebra $ \mathcal{O}_n / I_i $ with $i = 0,1$, we say that homomorphism $\phi :\mathcal{O}_n / I_1\to\mathcal{O}_n / I_2$ is homogeneous if $\phi$ preserves the weights of $\mathcal{O}_n$.
	 For the case where both $f_1$ and $ f_2$ are weighted homogeneous, a contact equivalence $(\phi,u)$ as described in (1) of the Mather-Yau theorem induces a homogeneous homomorphism by  truncating the higher-order terms.  Therefore, we obtain the equivalent condition of the Mather-Yau theorem:
\begin{enumerate}
\item[(4)] (Suppose that $f_1$ and $f_2$ are weighted homogeneous.) There exists some homogeneous automorphism $ \phi \in \Aut(\mathcal{O}_n) $ and a constant $ \lambda \in \mathbb{C}^* $ such that 
\[
	\phi(f_1) = \lambda \cdot f_2.  
\]  
\end{enumerate}

 In \cite{Saito1974Einfach},  Saito defined a simple elliptic singularity to be a normal surface
     singularity such that the exceptional set of the minimal resolution is a smooth
     elliptic curve, and classified those singularities that are hypersurface singularities into
     the following three weighted homogeneous cases: 
	 \begin{equation}\label{equ:definition}
		\begin{aligned}
			\tilde{E}_6: \;& x^3+y^3+z^3 + t xyz = 0, \quad \text{for $ t \in \mathbb{C}(\tilde{E}_6) $} ; \\
			\tilde{E}_7: \;& x^4+y^4+z^2 + t x^2 y^2 = 0,  \quad  \text{for $ t \in \mathbb{C}(\tilde{E}_7) $} ;\\
			\tilde{E}_8: \;& x^6+y^3+z^2 + t x^4 y = 0,\quad  \text{for $ t \in \mathbb{C}(\tilde{E}_8) $} ;
			\end{aligned}
	 \end{equation}
	 where the restrictions on $t$ ensure the singularities are isolated: 
	 \begin{align*}
		 \tilde{E}_{6}: t \in  \mathbb{C}(\tilde{E}_6) &= \{  t \in \mathbb{C} |    t^3 \not = 27 \}  ;\\
		 \tilde{E}_{7}: t \in  \mathbb{C}(\tilde{E}_7) & = \{  t \in \mathbb{C} |    t^2 \not = 4 \}  ;\\
		 \tilde{E}_{8}: t \in   \mathbb{C}(\tilde{E}_8) & = \{  t \in \mathbb{C} |   4 t^3 \not = -27 \}  . 
		\end{align*}
		Using a result of Noumi and Yamada \cite{noumi1997notes} on the flat structure on the space of
		versal deformations of these singularities, Strachan \cite{2010Simple}
		 explicitly constructed the $G$-functions for these families. Saito determined the  isomorphism  classes of these families by computing the $ j $-functions and concluded the criteria that two singular with coefficients $t$ and $s$ are contact equivalent if and only if $ j(\tilde{E}_i; t) =  j(\tilde{E}_i; s) $, 
		where
		\begin{align*}
			j(\tilde{E}_6;t) &= -\frac{t^3(t^3-216)}{1728 (t^3+27)^3} ; \\
			j(\tilde{E}_7;t) &= \frac{(12+t^2)^3}{108(t^2-4)^2} ; \\
			j(\tilde{E}_8;t) &= \frac{4t^3}{4t^3+27}.
		\end{align*} 
	
	For $i =6,7,8$, we denote by $ \mathcal{A}^k(\tilde{E}_i; t ) $ with $ t \in \mathbb{C}(\tilde{E}_i ) $ the corresponding $k$-th moduli algebra (resp. local function algebra) of $ \tilde{E}_i $-families listed above. With the help of the Mather-Yau theorem, Eastwood \cite{Eastwood2004Moduli} demonstrated explicitly that one can recover directly Saito's $j$-functions  from the zeroth moduli algebra $ \mathcal{A}^0(\tilde{E}_i; t ) $. 
	
	In this paper, we aim to 
	recover Saito's $j$-functions  from the $k$-th moduli algebra $ \mathcal{A}^k(\tilde{E}_i; t ) $ with $ k \geqslant 1 $ by applying the $k$-th version of the Mather-Yau theorem. Furthermore, we pose a finer question: how can we completely describe the homogeneous isomorphisms from $\mathcal{A}^k(\tilde{E}_i; t )$ to $ \mathcal{A}^k(\tilde{E}_i; s ) $? To understand the algebraic structure in this context, we introduce the groupoid $ \Grp^k(E_i)$ of simple elliptic singularity. Our complete description of the groupoid $ \Grp^k(E_i)$ provides more information than the results presented in \cite{Eastwood2004Moduli}.
	
	A (small) groupoid is a small category in which every morphism is invertible. The notion of groupoid generalizes  the concept of a group by replacing the binary operation with a partial function. In this paper, we are interesting in the groupoid \cite{higgins1971categories,brown2006topology} presented in the following manner.
	 \begin{defn}[Action Groupoid]

		If the group  $G $ acts on the set $ X $, then we can form the action groupoid $ \AcGrp(X| g_1, \ldots, g_n )$ (or transformation groupoid) as follows. 
		\begin{enumerate}
			\item The objects are the elements of $X$.
			\item For any two elements $x, y \in X $,  the morphisms from $x $ to $ y $ correspond to the elements $g \in G $ such that  $g x = y $.
			\item Composition of morphisms interprets the binary operation of $G$.
		\end{enumerate}
	 \end{defn}

	 \begin{defn}
		For $i = 6,7,8$ and $ k \in \mathbb{Z}_{\geqslant 0} \cup \{ \infty \}$, we define the groupoid $ \Grp^k(\tilde{E}_i) $  of simple elliptic singularities as follows.
		\begin{enumerate}
			\item The objects of $ \Grp^k(\tilde{E}_i) $ are precisely the coefficients $ t\in \mathbb{C}(\tilde{E}_i) $. 
			\item For $ t, s \in \mathbb{C}(\tilde{E}_i ) $, the morphisms of $ t $ to $s$ are the homogeneous isomorphisms from $ \mathcal{A}^k(\tilde{E}_i; t) $ to  $ \mathcal{A}^k(\tilde{E}_i; s) $ modulo the scalars in $\mathbb{C}^*$. 
		\end{enumerate}
	 \end{defn}

	 This paper devotes to the explicit computation of the groupoids $ \Grp^k(\tilde{E}_i) $. 
	 From the lifting property, any morphism of $ \Grp^k (\tilde{E}_i) $ lifts to an automorphism of weighted projective spaces. Up to the $\mathbb{C}^*$-action, the morphism is essentially contained in the weighted projective general linear group $
	 	\PGL_3^{\mathbf{w}}(\mathbb{C}) $, which is the quotient of the general linear group $\GL_3(\mathbb{C})$ by its subgroup of diagonal matrices 
	 \[
	 	\begin{pmatrix}
			\lambda^{w_1} & 0 & 0 \\
			0 & \lambda^{w_2} & 0 \\
			0 & 0 & \lambda^{w_3} 
		\end{pmatrix}
	 \]
	 for some weights $   (w_1, w_2 , w_3) =:\mathbf{w}$.
	 When $ k $ is large, we obtain 
	  \[ \Grp^k (\tilde{E}_i)  = \Grp^\infty (\tilde{E}_i) ,\]
	  since the morphism in $\Grp^k (\tilde{E}_i)$ preserves the defining functions (see Lemma \ref{lem:groupoid}).

	  For $k=0$ or $1$, $ \Grp^k (\tilde{E}_i) $ contains more morphisms due to the fact that 
	  $f \in mJ(f)$ when $f $ is weighted homogeneous.
	  Finally, we also note that there exist some coefficients $t \in \mathbb{C}(\tilde{E}_i)$ that behave weirdly, referred to as jump points according to references\cite{Seeley1990Variation, hussain2022new}.  The jump points are listed as follows.
	 \begin{enumerate}
		\item For the $ \tilde{E}_{6} $-family, $ t = 0 ,6 , 6 \rho, 6 \rho^2 $;
		\item For the $ \tilde{E}_{7} $-family, $ t = 0 ,\pm 6 $; 
		\item For the $ \tilde{E}_8 $-family, $t = 0 $.  
	 \end{enumerate}
	 Here and afterwards,  $ \rho$ denotes a primitive cubic root of unity. In the following subsections, we will formulate the main results in detail.

	 \subsection{Groupoid of the $ \tilde{E}_6 $-family}
     In \cite{Chen2000Algebraic}, Chen, Seeley and Yau provided a detailed characterization of homogeneous isomorphisms of the moduli algebras arising from the $\tilde{E}_6 $-family, i.e., the morphisms in $ \Grp^0(\tilde{E}_6)$.  Our newly elaborated technique for studying the isomorphisms of  $ \Grp^1(\tilde{E}_6)$ differs from the previous approach used in \cite{Chen2000Algebraic}. In fact, our idea is motivated by the work of Seeley and Yau \cite{Seeley1990Variation}, in which the authors studied the derivation Lie algebras $ \{L_t\} $ associated to the zeroth moduli algebras of $\tilde{E}_7 $ and $\tilde{E}_8 $-families. They proved the Torelli-type theorem saying that the Lie algebras $ L_s $ and $ L_t $ are isomorphic if and only if the corresponding singularities are contact equivalent to each other.  In \cite{benson1986lie}, the Lie algebras of simple elliptic singularity were computed along with several elaborate applications to deformation theory. 
	 
	For the $\tilde{E}_6 $-family, we consider the derivation Lie algebra, written as  $L^1(\tilde{E}_6; t )$, associated with the first moduli algebra $\mathcal{A}^1(\tilde{E}_6; t )$. By equipping $L^1(\tilde{E}_6; t )$ with a natural module structure over the moduli algebra, we discover that the invariants arising from the Lie algebra $L^1(\tilde{E}_6; t )$, together with 
	$ \mathcal{A}(\tilde{E}_6; t)$, are applicable in determining the isomorphisms of the $ \tilde{E}_6 $-family. It turns out that the groupoids $ \Grp^{0}(\tilde{E}_6) $ and $ \Grp^{1}(\tilde{E}_6) $ coincide indicating a simpler proof for 
	Chen-Seeley-Yau's result in \cite{Chen2000Algebraic}. However, we note that the Lie algebra $L^1(\tilde{E}_6; t )$ itself (without the module structure) is insufficient to determine the isomorphism classes of the $\tilde{E}_6$-family. In other words, there is no Torelli-type theorem for $L^1(\tilde{E}_6; t )$ analogous to that in  \cite{Seeley1990Variation}.  The failure of a Torelli-type theorem for the  zeroth Yau algebra can also be found in \cite{elashvili2006lie}.
	  
     We define the subgroup $ G $ of $ \PGL_3(\mathbb{C}) $ generated by the matrices 
     \begin{align*}
     	&	A_1:=
     	\begin{pmatrix}
     		\rho & 0 & 0 \\
     		0 & 1 & 0 \\
     		0 & 0 & 1  
     	\end{pmatrix}, 
     	A_2:=
     	\begin{pmatrix}
     		\rho & \rho^2 & 1 \\
     		\rho^2  & \rho & 1 \\
     		1 & 1 & 1  
     	\end{pmatrix}, 
     	A_3:=
     	\begin{pmatrix}
     		0 & 1 & 0 \\
     		1 & 0 & 0 \\
     		0 & 0 & 1  
     	\end{pmatrix},  
     	A_4:=
     	\begin{pmatrix}
     		1 & 0 & 0 \\
     		0 & 0 & 1  \\
     		0 & 1 & 0  
     	\end{pmatrix}.
     \end{align*}
Note that $ A_3 $ and $ A_4 $ generates the symmetric subgroup on three letters. 
 One can show by direct computation that $ G $ contains totally $ 216 $ matrices and all the entries belong to $ \{ 0, \rho ,  \rho^2, 1 \}$.
 Let $ H $ be the group of fractional linear transformations generated by 
 \[  P : t \mapsto \rho t \]
and 
\[  R : t \mapsto \frac{18-3t }{3 + t} .\]
 It can be checked that $ H $ is a group of order $ 12 $, isomorphic to the alternative group on four letters. In addition, there exists a group homomorphism $ \pi: G \to H $ which sends $ A_1, A_2,A_3,A_4$ to $ P, R, \Id , \Id $ respectively.  
    
    
	We restate and generalize the main result in \cite{Chen2000Algebraic} in the following form.
	\begin{thmletter}\label{thmA}
		(1) For the case $ k = 0$ or $ 1 $, the groupoid $ \Grp^k(\tilde{E}_6) $ restricted to  the objects $\mathbb{C}(\tilde{E}_6) \setminus \{ 0,6,6 \rho, 6 \rho^2\}$
		is represented by the action groupoid 
		 \[ 
		 \AcGrp( \mathbb{C}(\tilde{E}_6) \setminus \{ 0,6,6 \rho, 6 \rho^2\} \, | \, G) ,
		 \]
		 where $A_i$ acts on $ \mathbb{C}(\tilde{E}_6)$ through the morphism $ \pi $ defined above.
		
		 (2)For $ k \geqslant 2 $ or $ k = \infty $, we find that $ \Grp^k( \tilde{E}_6)$ equals the action groupoid  
		 $
			\AcGrp( \mathbb{C}(\tilde{E}_6)\, | \, G )$. 

	  (3)(The Failure of a Torelli-Type Theorem) The first Yau algebra of the $ \tilde{E}_6 $-family is independent on the parameter $t $ and gives rise to a continuous family of $ 11 $-dimensional representations of a solvable Lie algebra.
	\end{thmletter}
	We remark that for $k = 0$ or $1$, $
			\AcGrp( \mathbb{C}(\tilde{E}_6)\, | \, G )
		 $ is a subcategory of $ \Grp^k(\tilde{E}_6) $, but it is not fully faithful. 
	Indeed, it can be concluded that 
	$
		\Grp^k(\tilde{E}_6)(0,0)  
	$
	is the group generated by $A_2, A_3, A_4$ and the diagonal matrices. 
	As a consequence of Theorem \ref{thmA}, K.~Saito's $j$-function of the $\tilde{E}_6$-family can be obtained by observing that the integral functions on $\mathbb{C}(\tilde{E}_6) $ fixed by $G$ (or $H$ equivalently) are generated by $ j(\tilde{E}_6; t) $. 

	\subsection{Groupoid of the $ \tilde{E}_7 $-family}
	Consider the group $ H' $ consisting of the following six fractional linear transformations 
	\[
		t \mapsto \pm t, t \mapsto \pm \left( \frac{12-2t}{2+t } \right) , t \mapsto \pm \left(\frac{12 +2 t}{2-t} \right).
		\]
	Note that $ H' $ is isomorphic to the symmetric group $ S^3 $.
	According to \cite{Seeley1990Variation}, instead of $j$-function $ j(\tilde{E}_7;t) $, we have the criteria that the coefficient $t$ is equivalent to $ s $ in $\Grp^k(\tilde{E}_7) $ if and only if $s = h(t) $ for a unique $ h \in H' $.
	We define subgroup $ G'$ of $ \PGL_2^{(1,1)}(\mathbb{C}) $ generated by
	\[
		\begin{pmatrix}
			 1 & 1   \\
		i & -i  
		\end{pmatrix}, 
		\begin{pmatrix}
			1 & 0  \\
	   0 & i  
	   \end{pmatrix}, 
	   \begin{pmatrix}
		0 & 1  \\
	1 & 0  
	\end{pmatrix}.
	\]
	This group is to shown to preserve the polynomial $x^4 + y^4 $ and is isomorphic to the symmetry group $ S^4 $. Thus, there are totally $24$ elements of $ G' $ which are represented by the following matrices
	\[ 
		B_{\alpha, 0} := \begin{pmatrix}
					1 & 0   \\
			   0 & \alpha  
			   \end{pmatrix}, 
			 B_{0, \beta}  := \begin{pmatrix}
				0 & \beta  \\
			1 & 0   
		\end{pmatrix} , 
		B_{\alpha, \beta} := \begin{pmatrix}
		1 & \beta  \\
		\alpha & -\alpha \beta   
	\end{pmatrix} ,
	\]
	where $ \alpha $ and $ \beta $ range over $ \{ \pm 1 , \pm i \} $.

	Notice that $S^4$ has a Klein four-group as a proper normal subgroup, namely the even transpositions ${\Id, (1 2)(3 4), (1 3)(2 4), (1 4)(2 3)}$, with quotient $S^3$.
	In according to the mapping $S^4\to S^3$, there exists a homomorphism 
	\[ \pi' : G' \to H',  
	\]
	defined as 
	\begin{equation}\label{equ:pip}
	\begin{aligned}
		B_{\alpha, \beta} 
		 & \mapsto \left[t \mapsto \frac{\beta^2 (12 -2 \alpha^2 t )}{2 + t \alpha^2 } \right], \\
		 B_{\alpha,0},
		B_{0,\alpha}
		 & \mapsto  [t \mapsto \alpha^2 t].
	\end{aligned}
\end{equation}
	In particular, the kernel of $\pi'$ gives a subgroup \[ G'_0 = \{  B_{ 1,0 } = \Id ,B_{-1,0} ,B_{0, 1 },B_{0, -1 }\} \]
	which is isomorphic to the Klein four-group.
	Using the notation 
	\[ \tilde{B}_{\alpha, \beta , \gamma} := \begin{pmatrix}
		B_{\alpha,\beta} & 0 \\  
		0   & \gamma
	\end{pmatrix} 
	\]
	for $ \gamma \in \mathbb{C}^* $,
	we obtain the following result which parallels to Theorem \ref{thmA}:

	\begin{thmletter}\label{thmB}
		(1) For $ k =0,1 $, the groupoid $ \Grp^k(\tilde{E}_7)$ restricted to $\tilde{E}_7\setminus \{ 0, \pm 6\}$ 
		is identical to the action groupoid	
		\[ 
		\AcGrp( \mathbb{C}(\tilde{E}_7) \setminus \{ 0, \pm 6\} \, | \,  \{ \tilde{B}_{\alpha, \beta, \gamma}, \gamma \in \mathbb{C}^* \}  ) .
		\]
		 Here $\alpha $ and $ \beta $  take values from $ \{ 0,\pm 1, \pm i \} $, and they are not simultaneously zero; and the matrices $\tilde{B}_{\alpha, \beta, \gamma}$ act on the points of $\mathbb{C}(\tilde{E}_7) $ through $ \pi' $, restricting to the $(x,y)$ coordinates.

		(2)If $ k \geqslant 2 $ or $ k = \infty $, then
		 \[  
		 \Grp^k(\tilde{E}_7)(t, s ) = \begin{cases}
			\{  \Id ,\tilde{B}_{-1,0,1} ,\tilde{B}_{0, 1,1},\tilde{B}_{0, -1,1}\}  & \text{when $s = t$};\\
			\{ \tilde{B}_{\alpha,0,1},
		\tilde{B}_{0,\alpha,1} \}  & \text{when $s = \alpha^2 t$}; \\
		 \{ \tilde{B}_{\alpha,\beta, \pm \sqrt{ 2 + \alpha^2 t}} \} & \text{when $s =  \frac{\beta^2 (12 -2 \alpha^2 t )}{2 +  \alpha^2t } $}. \\
		 \end{cases} 
		 \]
		(3) (Torelli-Type Theorem)For $t \in \mathbb{C}(\tilde{E}_7)\setminus \{ 0, \pm 6 \} $, the first Yau algebra determines the isomorphism class of the $ \tilde{E}_7 $-family.
	\end{thmletter}
	Similar to the $\tilde{E}_6$-family, the matrices of the form $ \tilde{B}_{\alpha, \beta, \gamma} $ generate a proper sub-groupoid of $ \Grp^k (\tilde{E}_7) $ when $k= 0$ or $1$. The integral functions on $ \mathbb{C}(\tilde{E}_7) $ fixed by $ \tilde{B}_{\alpha, \beta, \gamma} $ (or $H'$ equivalently)
	are generated by $ j(\tilde{E}_7;t)$.

	\subsection{Groupoid of the $\tilde{E}_8 $-family}
	The criteria for the isomorphism class of the $\tilde{E}_8 $-family is much simpler. It follows from the $j$-function $ j(\tilde{E}_8; t)$ that  $t$ is isomorphic to $s$ in $ \Grp^k(\tilde{E}_8)$ if and only if $s =at$ with $a^3 =1$. Let $ C_{ \rho^i, \gamma} $ be the matrix of the form 
	\[ \begin{pmatrix}
		1 & 0 & 0 \\
		0 & \rho^i & 0 \\
		0 & 0 & \gamma
	\end{pmatrix}
	\]
	with $ \gamma \in \mathbb{C}^* $. The collection of matrices $  \{ C_{ \rho^i, \gamma} \} $ with $i=0,1,2$, $c \in \mathbb{C}^*$ forms a subgroup of $ \PGL_3^{\mathbf{w}}(\mathbb{C}) $ with $ \mathbf{w} = (1,2,3) $.  
	For $ k = 0$ or $1$, the matrix $ C_{\rho^i, \gamma}$ induces an isomorphism from $ \mathcal{A}^k( \tilde{E}_8, t)$ to $ \mathcal{A}^k( \tilde{E}_8, \rho^i t)$, represented by the transformation
	\[
		x= x', y =\rho^i y' , z= \gamma z' .
	\]
	\begin{thmletter}\label{thmC}
	(1) For $ k = 0 $ or $ 1$,the groupoid $ \Grp^k(\tilde{E}_8 )$ is given by the union of the action groupoid  
	\[
		\AcGrp( \mathbb{C}(\tilde{E}_8) \setminus \{ 0 \} \, |\, \{ C_{\rho^i, \gamma} | \gamma \in \mathbb{C}^* \} )
	\]
	and the group of diagonal matrices (as a groupoid with a single object  $ 0 \in \mathbb{C}(\tilde{E}_8) $).

	(2) While for $k \geqslant 2 $ or $ k = \infty $, the groupoid $ \Grp^k(\tilde{E}_8 )$ is given by the action groupoid  
		$
			\AcGrp( \mathbb{C}(\tilde{E}_8) \,| \, \{ C_{\rho^i, 1} \} )		$.
	\end{thmletter}

	\section{Preliminary and notations}
   \subsection{Homogeneous isomorphism}
   In this paper, we denote by $ \mathcal{O}_n$ the formal local ring $\mathbb{C}[[x_1,\ldots,x_n]]$, and assume that the weight of $ x_i $ equals $w_i $. For simplicity, we assume the weights $ w_i $ are positive integers. 
   Given $ I_t $ and $ I_s  $, the two homogeneous ideals of  $ \mathcal{O}_n $, i.e., the generators in $I_t$ and $I_s$ are weighted homogeneous, we consider the graded analytic algebras $ \mathcal{O}_n / I_t $ and $ \mathcal{O}_n / I_s $.  By the well-known lifting lemma, the isomorphism from  $ \mathcal{O}_n / I_t $ to $ \mathcal{O}_n / I_s $ is given by 
   $ \phi  \in \Aut (\mathcal{O}_n) $ such that
   \[ \phi (I_t) \subseteq I_s, \text{and } \phi^{-1}(I_s) \subseteq I_t. 
   \]
   One can express $ \phi : \mathcal{O}_n \to \mathcal{O}_n$ as
  \[  x_j \mapsto \phi^{(0)}_j(x_1,\ldots, x_n) + \phi^{(+)}_j(x_1,\ldots, x_n)  \text{ for $j =1, \ldots, n $}, \]
  such that the monomials in $ \phi^{(0)}_j(x_1,\ldots, x_n) $ have weight $ w_j $, while those in $ \phi^{(0)}_j(x_1,\ldots, x_n) $ have weight $ > w_j $.
  The homogeneous component $\phi^{(0)} $ of $\phi$, given by $\phi^{(0)} = (\phi_j^{(0)})_{j}$, forms an isomorphism if and only if $\phi$ is an isomorphism. Consequently, each isomorphism from $\mathcal{O}_n / I_t$ to $\mathcal{O}_n / I_s$ induces a homogeneous isomorphism.

A $\mathbb{C}^*$-action is present on $\mathcal{O}_n$ defined as:
\[
   \lambda \in \mathbb{C}^* : x_j \to \lambda^{w_i} x_j,
\]
resulting in a trivial automorphism of $\mathcal{O}/I_t$. Moreover, if $\phi(x_1,\ldots,x_n)$ represents a homogeneous isomorphism, then $\phi(\lambda^{w_1}x_1,\ldots,\lambda^{w_n}x_n)$ also signifies a homogeneous isomorphism. This establishes the $\mathbb{C}^*$-action within the set of homogeneous isomorphisms.

\subsection{Yau algebra with graded module structure}
 A graded Lie algebra is an ordinary Lie algebra $g $ together with a gradation of vector spaces
 \[ 
	\mathfrak{g} = \oplus_{i \in \mathbb{Z} } \mathfrak{g}_i \]
such that the Lie bracket respects this gradation:
\[ 
 [\mathfrak{g}_i ,\mathfrak{g}_j ] \subseteq \mathfrak{g}_{i + j } .
 \]
 
 \begin{defn}
	 For an isolated hypersurface singularity  $ \mathcal{V}_f $ determined by a polynomial $ f = f(x_1, \ldots ,x_n) $, we define the $ k $-th Yau algebra of $ \mathcal{V}_f $ by
	 \[  L^{k} (\mathcal{V}_f): = \Der(\mathcal{A}^{k}(\mathcal{V}_f),\mathcal{A}^{k}(\mathcal{V}_f)) ; \]
	 namely the derivation Lie algebra of the $k$-th moduli algebra.
 \end{defn}
 We need to introduce a graded $ \mathcal{A}^{k}(\mathcal{V}_f)$-module structure of $L^{k} (\mathcal{V}_f)$.
Denote by $\partial_i $ the partial derivation with respect to $x_i $.
By definition, a derivation $ \delta \in L^{k} (\mathcal{V}_f) $ is of the form
\[ \delta =  \sum_{i=1}^{n} a_i \partial_i, \]
where the coefficients $a_i \in \mathcal{A}^{k}(\mathcal{V}_f)$ satisfy the condition
\[
 \delta \cdot \langle f, m^k J(f) \rangle \subseteq \langle f, m^k J(f) \rangle. 
\]
In this way, we obtain the embedding map of $ \mathcal{A}^{k}(\mathcal{V}_f) $-modules:
\begin{equation}\label{equ:Lk}
  L^{k} (\mathcal{V}_f) \to  \mathcal{A}^{k}(\mathcal{V}_f) \langle \partial_{1} , \ldots, \partial_{ n} \rangle 
\end{equation}
where $ \mathcal{A}^{k}(\mathcal{V}_f) \langle \partial_{1} , \ldots, \partial_{ n} \rangle $ denotes the free $\mathcal{A}^{k}(\mathcal{V}_f)$-module with generator $ \partial_1, \ldots, \partial_n$.

 With the assumption that $ f $ is a weighted homogeneous polynomial, the algebra $ \mathcal{A}^{k}(\mathcal{V}_f) $ admits a graded algebra structure since 
 the ideal $\langle f, m^k J(f) \rangle $ is homogeneous.
 For a monomial $\mathcal{M}$ in $ \mathcal{A}^{k}(\mathcal{V}_f) $, we define the notion of degree  as follows:
\[
 \deg (\mathcal{M}\partial_{x_i}) = \weight(\mathcal{M}) - \weight(x_i). 
\] 
Based on this grading, the free 
$\mathcal{A}^k(\mathcal{V}_f) $-module $\mathcal{A}^k(\mathcal{V}_f)\langle \partial_{1}, \ldots, \partial_{ n} \rangle$ is regarded as a graded Lie algebra. Similarly, the $k$-th Yau algebra  $ L^k (\mathcal{V}_f)$  inherits a graded Lie algebra structure.

We utilize the crucial observation that a homogeneous isomorphism of $k$-th moduli algebras leads to an isomorphism of the corresponding $k$-th Yau algebras.
For weighted homogeneous polynomials $f_1$ and $f_2$, a homogeneous isomorphism
\[ 
\phi: \mathcal{A}^k(\mathcal{V}_{f_1})   \to  \mathcal{A}^k(\mathcal{V}_{f_2}) \]
 induces a Lie-algebraic homomorphism 
\[
	\phi_*: L^k (\mathcal{V}_{f_1})\to L^k (\mathcal{V}_{f_2}).
\]
Such homomorphism $\phi_*$ is $\phi$-equivalent meaning that 
\[
 	\phi_*(g \E_1) = \phi(g) \phi_*(\E_1) 
\]
and 
\[
 	\phi_*[\E_1, \E_2] = [\phi_*\E_1 , \phi_*\E_2 ]
\]
hold for $ g \in \mathcal{A}^k(\mathcal{V}_{f_1})$ and $ \E_1, \E_2 \in L^k (\mathcal{V}_{f_1})$.
\subsection{Notation of homomorphisms for Lie algebra}

When focusing on the case where  $ \mathcal{V}_f $ from $\tilde{E}_i$-families, we consistently employ the subsequent notations to denote homogeneous isomorphisms of the $k$-th moduli algebras of simple elliptic singularities throughout this paper.

 \begin{notation}
	Suppose that $(x', y', z')$ are the coordinates of $\mathcal{A}^k(\tilde{E}_i; s)$ and $(x, y, z)$ are the coordinates of $\mathcal{A}^k(\tilde{E}_i; t)$. We denote a homogeneous isomorphism $\phi$ from $\mathcal{A}^k(\tilde{E}_i, t)$ to $\mathcal{A}^k(\tilde{E}_i, s)$ through coordinate transformations as:
	\begin{equation}\label{equ:phi_i}
		\begin{cases}
			x = \phi_1(x',y',z');\\
			y = \phi_2(x',y',z');\\
			z = \phi_3(x',y',z'),
			\end{cases}
	\end{equation}
	where $\phi_i $'s are homogeneous functions.
\end{notation}
Specifically, when all $w_i = 1$, $\phi$ uniquely corresponds to a nonsingular linear transformation of $\mathbb{C}^3$. One can represent $\phi$ as a linear transform
\begin{equation}\label{equ:phi}
     \begin{pmatrix}
	x \\
	y \\
	z  
\end{pmatrix}=
\begin{pmatrix}
	a_1& a_2& a_3 \\
	b_1& b_2& b_3 \\
	c_1& c_2& c_3  
\end{pmatrix}
\begin{pmatrix}
	x '\\
	y '\\
	z'  
\end{pmatrix}.
\end{equation}
Its inverse can be expressed as  
\begin{equation}\label{equ:inverse}
	\begin{pmatrix}
		x '\\
		y' \\
		z  '
	\end{pmatrix}=
	\begin{pmatrix}
		a_1'& a_2'& a_3' \\
		b_1'& b_2'& b_3' \\
		c_1'& c_2'& c_3'  
	\end{pmatrix}
	\begin{pmatrix}
		x \\
		y \\
		z  
	\end{pmatrix}.
\end{equation}  
We define $ \mathcal{A}^k(\tilde{E}_i;t) \langle \partial_x, \partial_y, \partial_z \rangle  $ to be the free $ \mathcal{A}^k(\tilde{E}_i; t) $-module generated by $ \partial_x, \partial_y,\partial_z $. 

As in \eqref{equ:Lk}, we possess the embedding map
\[
L^k(\tilde{E}_i; t) \to  \mathcal{A}^k(\tilde{E}_i;t) \langle \partial_x, \partial_y, \partial_z \rangle   .
\]
According to Leibniz's Rule, we derive the formulas for the differential of $\phi$:
 \begin{equation}\label{equ:partial}
\begin{aligned}
	 \phi_* (\partial_{x}) &= \partial_{x} (x') \partial_{x'} +\partial_{x} (y') \partial_{y'} +\partial_{x} (z') \partial_{z'} = a'_1 \partial_{x'} + b'_1 \partial_{y'} + c'_1 \partial_{z'} ;\\
	  \phi_* (\partial_{y}) &= \partial_{y} (x') \partial_{x'} +\partial_{y} (y') \partial_{y'} +\partial_{y} (z') \partial_{z'} = a'_2 \partial_{x'} + b'_2 \partial_{y'} +c'_2 \partial_{z'} ;\\
	   \phi_* (\partial_{z}) &= \partial_{z} (x') \partial_{x'} +\partial_{z} (y') \partial_{y'} +\partial_{z} (z') \partial_{z'} = a'_3 \partial_{x'} + b'_3 \partial_{y'} +c'_3 \partial_{z'} .
\end{aligned}
\end{equation}
 In this way, we obtain a $\phi$-equivalent map  
 \begin{align*}
	\phi_*:  & L^k(\tilde{E}_i; t)  \to L^k(\tilde{E}_i; s)\\
	 \phi_*(x^ay^bz^c \partial_i) = & \phi_1(x',y',z')^a\phi_2(x',y',z')^b\phi_3(x',y',z')^c \phi_*(\partial_i),
 \end{align*}
 which preserves the Lie structures.
 \subsection{Basis property of groupoid} 
 We examine some fundamental properties of the simple elliptic groupoid $\Grp^k(\tilde{E}_{i}) $ associated to $\mathcal{A}^k(\tilde{E}_i; t) $. Let $ f_t $ be the defining function of the $\tilde{E}_i$-family as described in Equation \eqref{equ:definition}. Recall that the $ k $-th moduli algebra is defined by the formal local ring $ \mathcal{O} $ modulo the ideal 
 \[
	 I_k(t) = \langle f_t, m^k J(f_t)\rangle.
 \]
 In this context, a morphism $\phi$ of $ \Grp^k(\tilde{E}_i)(t,s)$ is expressed as an automorphism in $
	  \PGL_3^{\mathbf{w}}(\mathbb{C}) $ such that $ \phi(I_k(t)) = I_k(s) $.
\begin{lem}\label{lem:groupoid}
	\begin{enumerate}
		\item The groupoid $\Grp^\infty(\tilde{E}_{i}) $ is a sub-groupoid of  $\Grp^k(\tilde{E}_{i}) $ for each $ k \geqslant 0 $. 
	\item When $ l  $ is sufficiently large, we have 
	\[   \Grp^k(\tilde{E}_{i}) =\Grp^\infty(\tilde{E}_{i}) \]
	for $ k \geqslant l $. 

	\item The groupoid $\Grp^0(\tilde{E}_{i}) $ is a sub-groupoid of  $\Grp^1(\tilde{E}_{i}) $. 
	\end{enumerate}
\end{lem}
\begin{proof}
	
	(1) The first assertion is deduced from the fact that the isomorphism of the coordinate ring induces weighted homogeneous isomorphisms of the $k$-th moduli algebras.

	(2)
	For sufficiently large $k$, we can assume that the weights of the generators in $m^k J(f_t)$ exceed the weight of $f_t$. Consequently, the homogeneous polynomial $f_t$, as a generator of the defining ideal $I_k(t)$ of the algebra $ \mathcal{A}(\tilde{E}_i; t)$,  has the minimal weight among the generators. Therefore, each morphism in $ \Grp^k(\tilde{E}_i)$ yields a transformation $ \phi $ of $\mathcal{O}_n$ which satisfies the equation
	\[ \phi f_t(x,y,z) = \lambda \cdot f_s(x',y',z') \]
	for some constant $ \lambda $. This relationship yields that $\phi$ is  essentially a weighted homogeneous isomorphism between $k$-moduli algebras.

	(3) When $k=0,1$, we have $ f_t \in mJ(f_t) $ for any weighted homogeneous polynomial $f_t$. Then 
	\[
	I_1(t) = \langle f_t, m J(f_t) \rangle =\langle  m J(f_t) \rangle  = m I_0(t).
	\]
	Note that a homogeneous isomorphism $ \phi $ preserves the maximal ideal $m$. Therefore, a homogeneous isomorphism $ \phi \in \Grp^{0}(\tilde{E}_i)$ is also an isomorphism in the context of the first moduli algebras.
\end{proof}

\section{Groupoid of the $\tilde{E}_8$-family}
In this section, we focus on the isomorphisms of the $ \tilde{E}_8 $-family. Recall that the $\tilde{E}_8 $-family is defined by the polynomial 

\[
f_t = x^6 + y^3 + z^2 + t x^4 y = 0, \quad \text{where } 4 t^3 \neq -27,
\]
which is weighted homogeneous of type $(1, 2, 3; 6)$. By definition, the \( k \)-th moduli algebra takes the form
\[
\mathcal{A}^k(\tilde{E}_8; t) = \mathbb{C}[[x,y,z]] / I_k(t),
\]
where 
\[
I_k(t) = \left\langle x^6 + y^3 + z^2 + t x^4 y, \, m^k (3x^5 + 2t x^3 y), \, m^k(3y^2 + t x^4), \, m^k z \right\rangle.
\]
Our goal is to determine the isomorphism classes of $\mathcal{A}^k(\tilde{E}_8 ; t) $ for various values of $t $.
 Since the monomials $y , x^2 $ are both have weight $2$ and $x$ is the unique monomial of weight $1$, any homogeneous automorphism $ \phi $ of $ \tilde{E}_8$-families can be expressed as 
\[
\begin{cases}
   x= \lambda  x',\\
y= \lambda ^2  ( ay'+bx^{\prime 2} ),\\
z = \lambda^3 c z ' + d x'^3 + e x' y'
\end{cases}	
\lambda,a ,c \in \mathbb{C}^*, b \in \mathbb{C} 
\]
in accordance with the notation in \eqref{equ:phi_i}.
 Up to scaling, we can always assume that $ \lambda = 1 $. 
 The possible generators with weight less than $k+3$ are $x^k z $ and $ f_s $ (when $3 < k$).
From the fact
\[ x^k z = 0 \mod I_k (t), \]
we have 
\[
   x'^k (z' + d x'^3 + e x' y') =0 \mod ( x'^k z' ,  f_s ).
\]  
This implies that $ d = e = 0 $. 

Similarly, we have
\[ x^k (3 y^2 + t x^4 ) = 0 \mod I_k (t), \] 
and the possible generators in $I_k(s)$ with weight $\leqslant
k+4 $ are $x^k (3 y^2 + t x^4) , z $ and $   f_s $.
\[
x'^k (3 (ay'+bx'^2)^2 +t x'^4) =0 \mod (  f_s,  z,  x'^k (3 y'^2 + s x'^4 ))
\]
This implies that $ b = 0 $. 
In conclusion, we find that the automorphism $ \phi $ is represented by a diagonal matrix.
\subsection{Case $ k \geqslant 2 $}
The generator in $I_k(t)$ with weight less than or equal to $6$ are $f_t$,  $
  x^3 z$ (for $k=3$), and $ xy z $, $ x^2(3y^2 +tx^4)$ (for $k =2  $).
From this, we obtain the relation 
\[
x^6+y^3+z^2 + t x^4 y = c^2 (x'^6+y'^3+z'^2 + s x'^4 y' ) \mod (f_s, x'^3z' , x'y' z',x'^2(3y^2 +s x'^4) ).
\]
This simplifies to:
\[
x^6+y^3+z^2 + t x^4 y = c^2 (x'^6+y'^3+z'^2 + s x'^4 y' ) .
\]
Since $ x = \phi_1(x') = x' $, we find that $c =1 $. Hence, we have 
\[  a =\rho^i , b = 0 ,c =1 ,s = t \rho^i .
\]
This completes the proof of the case $k=2$ in Theorem \ref{thmC}.

\subsection{Case $ k = 0  $}
From relations
\[
   \begin{cases}
	   3x^5 + 2t x^3y = 3 y^2 + t x^4=0  \mod(I_0(t))  \\
	   3x^{\prime 5} + 2s x^{\prime 3}y = 3 y^{\prime 2} + s x^{\prime 4}=0 \mod(I_0(s)) ,
   \end{cases}	
\]
we obtain 
\begin{align*}
	3x^5 + 2t x^3y =& 3x^{\prime 5} + 2s x^{\prime 3}y \\
	3 y^2 + t x^4=& 3 y^{\prime 2} + s x^{\prime 4}  .
\end{align*} 
Therefore, we have 
\[
	t = s a , \text{ and } t = \frac{s}{ a^2}.  	
\]
For $ t \not = 0$, we obtain  
\[
    a^3 =1, s= ta .
\]
Note that in this case, there is no essential restriction on  $ \phi(z) = c z ' $, allowing $c$ to be any arbitrary element in $\mathbb{C}^*$. For $t = 0 $, then we obtain $s =0 $ and there are no essential restrictions on $ a  $ and $c $. 
\subsection{ Case $ k = 1 $}
This case is essentially the same with $  k = 0 $ by replacing $ J(f_t) $ with $ m J(f_t) $. 
This completes the proof of Theorem \ref{thmC}.
  \section{Groupoid of the $ \tilde{E}_6 $-family}
   Now we turn  our attention to the study of simple elliptic singularity $\tilde{E}_6 $.
  For the parameter $ t \in \mathbb{C}(\tilde{E}_6)$ associated with the $ \tilde{E}_6 $-family, the isolated hypersurface singularity is defined by the polynomial
  \[ f_t : = x^3 + y^3 + z^3 +t xyz. \] 
 Note that the Jacobi ideal of $f_t$ is given by
\[ J(f_t) = \langle  3 x^2 + t yz  , 3 y^2 + t xz , 3 z^2 + tx y  \rangle . \]
We have
\[ \mathcal{A}^{0}(\tilde{E}_6; t) = \mathbb{C} [[  x,y,z ]]  /  \langle 3 x^2 + t yz  , 3 y^2 + t xz , 3 z^2 + tx y \rangle , \]
and
\[ \mathcal{A}^{1}(\tilde{E}_6; t) = \mathbb{C} [[  x,y,z ]]  /   \langle   3 x^3 + t xyz  , 3 y^3 + t xyz , 3 z^3 + tx yz, x^2 y , x^2z ,y^2 x, y^2z ,z^2x, z^2y  \rangle  . \]
For $ k \geqslant 2 $, it is easy to see that
\[ m^{k+2} \subseteq m^{k} J(f_t) \subseteq  m^{k+2} . \]
Hence, $ m^{k} J(f_t) = m^{k+2} $. This implies that for $ k \geqslant 2$, 
\[ \mathcal{A}^{k}(\tilde{E}_6; t)  = \mathbb{C} [[x,y,z]] / \langle f_t , m^{k} J(f_t) \rangle  =  \mathbb{C} [[x,y,z]] / \langle f_t , m^{k+2}\rangle. \] 
Applying (2) in Lemma \ref{lem:groupoid} it implies that for $ k \geqslant 2 $,  
\begin{equation}\label{equ:GrpE6}
	\Grp^k(\tilde{E}_6) = \Grp^\infty(\tilde{E}_6) .
\end{equation}
  
In order to prove Theorem \ref{thmA}, we need to study the first Yau Algebra of the $ \tilde{E}_6 $-family, defined as 
 \[  L^1(\tilde{E}_6; t) := \Der ( \mathcal{A}^1(\tilde{E}_6; t), \mathcal{A}^1(\tilde{E}_6; t)) . \]
 By direct calculation, we find that for the case $t \in \mathbb{C}(\tilde{E}_6)\setminus \{ 0, 6, 6 \rho , 6 \rho^2 \} $, the algebra
 $  L^1(\tilde{E}_6; t) $ is $22$-dimensional, with a $ \mathbb{C}$-linear basis represented as 
\[  \{x \partial_x + y \partial_y + z \partial_z \} \cup m^2 \mathcal{A}^1 (\mathcal{V}_t ) \langle \partial_{x}, \partial_{y}, \partial_{z} \rangle . \]
  One may choose the basis of $ \mathcal{A}^1(\tilde{E}_6; t) $:
  \[  \{ 1, x, y, z,  x^2, y^2, z^2, x y, xz, yz, xyz \}  
  \]  
with multiplication rules:
\begin{align*}
x^3 &= y^3 = z^3 = - \frac{t}{3} xyz, \\
x^2 y & = x^2 z =y^2 x = y^2 z = z^2x = z^2 y = 0.
\end{align*}
Then the set 
  \begin{align*}
  	\{ x \partial_x + y \partial_y + z \partial_z \} \cup \{ x^2 \partial_{i}, y^2 \partial_{i}, z^2 \partial_{i}, x y \partial_{i}, xz \partial_{i}, yz \partial_{i}, xyz \partial_{i} \},
  \end{align*}
  where $ \partial_{i} = \partial_{x}, \partial_{y}, \partial_{z} $, forms a basis of $ L^1(\tilde{E}_6; t) $ (see Section \ref{sec:representation} for another complete list of bases).
  In the case of jump points, i.e., $ t \in \{ 0, 6, 6 \rho , 6 \rho^2 \} $, the dimension of $  L^1(\tilde{E}_6; t) $ equals $24$.

  Through direct computation, we find that each matrix $ A_i $ in Theorem \ref{thmA} is a morphism in $ \Grp^k(\tilde{E}_6 ) $ for $ k \geqslant 0 $ and induces an isomorphism of the first Yau algebra. Based on this observation, the assertions (1) and (2) of Theorem \ref{thmA} can be divided into the three assertions of the following Lemma.
\begin{lem}\label{lem:assertion}
\begin{enumerate}
	\item  The action groupoid $ \AcGrp(\{ 0, 6, 6\rho, 6\rho^2 \} \, |\, G ) $ is a sub-groupoid of $ \Grp^k(\tilde{E}_6) $. In particular, $ \AcGrp(\{ 0, 6, 6\rho, 6\rho^2 \} \, |\, G ) $ becomes a fully faithful sub-groupoid of $ \Grp^k(\tilde{E}_6) $ when $ k \geqslant 2 $. 

	\item When $t ,s \in \mathbb{C}(\tilde{E}_6)\setminus \{ 0, 6, 6 \rho , 6 \rho^2 \} $, the equality 
	\[ \Grp^k(\tilde{E}_6)(t,s)=\Grp^\infty(\tilde{E}_6)(t,s)
	\] holds for $ k \geqslant 0 $.

	\item When $ t ,s\in \mathbb{C}(\tilde{E}_6) \setminus \{ 0, 6,6\rho,6\rho^2 \} $, each morphism in $\Grp^1(\tilde{E}_6)(t,s)$ is generated by $ A_1,\ldots, A_4  $.
\end{enumerate}
\end{lem}

\begin{proof} [Proof of the first assertion]
	It is straightforward to verify that the group $G$ induces morphisms of $ \Grp^k(\tilde{E}_6)$, and that the objects $ 0, 6, 6\rho, 6\rho^2 $ are stable under $ G $.

	For $ k\geqslant 2 $, the morphisms in $\Grp^{k}(\tilde{E}_6)(0,0) $ preserve the cubic polynomial  
	\[
		f_0 = x^3 + y^3 + z^3. 
	\]
	Therefore, $\Grp^{k}(\tilde{E}_6)(0,0) $ is represented by scaling matrices (equivalent to $\Id$ in $ \PGL_3(\mathbb{C})$) and permutations of three letters. By composing with  $ A_1^i A_2 $, it becomes apparent that  $\Grp^{k}(\tilde{E}_6)(t,s) $ for $ t ,s \in \{ 0, 6, 6 \rho , 6 \rho^2 \} $ is generated by $ A_1,\ldots,A_4$.
  \end{proof}

  \subsection{The second assertion}
  \begin{notation}\label{Nt:UV}
	\begin{enumerate}
	\item For an element $ \w $  in $ \mathcal{A}^1(\tilde{E}_6 ; t) $ of weight one, we introduce the symbol
	\[
		\F_{t}^{(x,y,z)}(\w) = f_t(\alpha, \beta ,\gamma) 
	\]
	where $ \alpha, \beta, \gamma $ are coefficients satisfying $ \w = \alpha x + \beta y + \gamma z $. Additionally, we will utilize the notation $ \F_{s}^{(x',y',z')}(\w)$ in reference to alternative coordinates $x',y',z'$.

	\item For linear subspace $ U, V \subseteq L^1(\tilde{E}_6; t) $, we denote by $ [U,V ]_t $ the vector space spanned by all elements $[u,v]_t$ where $u \in U $, and $v \in V$. For a vector $v$, we also define $[v,U]_t$ in the obvious manner. 
	
	\end{enumerate}
  \end{notation}      

  To prove the second assertion of Lemma \ref{lem:assertion}, we need the following technique lemma characterizing the relation between $ \Grp^1(\tilde{E}_6) $ and $ \Grp^\infty(\tilde{E}_6) $.
  \begin{lem}\label{lem:transpose}
	Assume that   $\phi$ is contained in $ \Grp^1(\tilde{E}_6)(t,s) $ with $ t,s \in \mathbb{C}(\tilde{E}_6)\setminus \{ 0, 6 ,6 \rho , 6 \rho^2 \} $ and $ \w $ is an element of weight one in $\mathcal{A}(\tilde{E}_6;t)$.
	\begin{enumerate}
		\item  The quantity
		\[
			\F_{\phi}:= \frac{  \F_{\frac{-18}{s}}^{(x',y',z')}( \phi(\w))   }{ \F_{\frac{-18}{t}}^{(x,y,z)}(\w)   }
		\]
		is independent on the choice of $ \w $ (but depends on $t,s$).
		\item  The transpose of the matrix $\phi$ is contained in $ \Grp^{\infty}(\tilde{E}_6)(\frac{-18}{s}, \frac{-18}{t}) $. 
	\end{enumerate}
  \end{lem}
	\begin{proof}
		(1)	Denote by $ [-,-]_t$ the Lie bracket of $ L^1(\tilde{E}_6;t )$. 
		Let $D_t$ be the $\mathbb{C}$-linear space with linear basis $ \partial_x, \partial_y, \partial_z $. 

		For an element $ \w = \alpha x + \beta y + \gamma z  $  in $ \mathcal{A}^1(\tilde{E}_6 ; t) $ of weight one, we define $ \mathcal{H}_ \w : = \w  ( x \partial _x +y  \partial_y + z \partial_z ) $.   Then 
		\[       [  \mathcal{H}_\w ,   L^1(\tilde{E}_6 ; t)  ]_t = \langle xyz (-\beta \partial_{x} + \alpha \partial_{y} ),  xyz(-\gamma \partial_{x} + \alpha \partial_{z}) \rangle. \]
		We obtain a linear subspace of $ D_t $ parameterized by $ \w $:
		\[  D_t''(\w) := \langle -\beta \partial_{x} + \alpha \partial_{y}, -\gamma \partial_{x} + \alpha \partial_{z}\rangle . \] 
		For arbitrary $ \partial = \xi_1 (-\beta \partial_{x} + \alpha \partial_{y} ) + \xi_2(-\gamma \partial_{x} + \alpha \partial_{z})  \in D_t''(\w) $ with $ \xi_1, \xi_2 \in \mathbb{C} $, we have
		\begin{align}
			&[ \w^{ 2} \partial ,   \mathcal{H}_ \w ]_t  \nonumber \\
			& = [ \w^{ 2} \xi_1 (-\beta \partial_{x} + \alpha \partial_{y} )+\w^{ 2} \xi_2 (-\gamma \partial_{x} + \alpha \partial_{z})  , \w (x\partial_x + y \partial_y + z \partial _z ) ]_t  \nonumber  \\ 
			&=\frac {t}{3} \F_{\frac{-18}{t}}(\w)  \cdot xyz \partial  . \label{eq:w^2}
		\end{align} 
		Assume that $ \phi $ is represented by the matrix \eqref{equ:phi} with inverse \eqref{equ:inverse}.
		Let us investigate the images of quantities $ \w $, $ \mathcal{H}_\w$, $ D_t''(\w) $ and $ [ \w^{ 2} \partial ,   \mathcal{H}_\w ]_t $   
		 under the isomorphism $ \phi $.
\begin{itemize}
	\item For the image of $ \w $, we obtain
\begin{equation} \label{equ:phiw}
	  \phi (\w) = (\alpha a_1+ \beta b_1 +\gamma c_ 1  )x'+ (\alpha a_2+ \beta b_2+\gamma c_ 2 ) y' +  (\alpha a_3+ \beta b _3+\gamma c_3 )z '   . 
\end{equation}
 
\item For $ \mathcal{H}_\w$, we have 
		\begin{equation*} 
			   \phi_* (\mathcal{H}_ \w)  =  \phi (\w)    ( x' \partial _{x'} +y'  \partial_{y'} + z'\partial_{z'}) .
		\end{equation*}
\item For $D_t''(\w)$, we compute the Lie bracket of $ \phi_* (\mathcal{H}_ \w) $ and $ L_s^1 $ to obtain
\begin{equation}\label{equ:Dtw}
\begin{aligned}
	\phi_* (D_t''(\w))   = \langle &  -(\alpha a_2+ \beta b_2+\gamma c_ 2 ) \partial_{x'} + (\alpha a_1+ \beta b_1 +\gamma c_ 1  ) \partial_{y'} ,  \\
	&- (\alpha a_3+ \beta b _3+\gamma c_3 ) \partial_{x'} + (\alpha a_1+ \beta b_1 +\gamma c_ 1  ) \partial_{z'} \rangle  .
\end{aligned}
\end{equation}
\item If $ \partial_*= \phi_*(\partial)   \in \phi_*( D_t''(\w) ) $, then 
\begin{equation}\label{eq:phi_w^2}
	[ \phi_*(\w^{ 2} \partial ), \phi_*( \mathcal{H}_ \w) ]_s =[ \phi (\w^{ 2}) \partial_* , \phi_*( \mathcal{H}_\w) ]_s = \frac {s }{3} \F_{\frac{-s}{18}}^{(x',y',z')}(\phi(\w)) x'y'z' \partial_*.
\end{equation}  
\end{itemize}

Note that $ \phi_* $ preserves the Lie brackets.
Combining Equations (\ref{eq:w^2}) and (\ref{eq:phi_w^2}), we find
		\begin{align*}
			\frac {t}{3} \F_{\frac{-18}{t}}^{(x,y,z)}(\w)  \cdot \phi_*( xyz )\partial_* &= \frac {t}{3} \F_{\frac{-18}{t}}^{(x,y,z)}(\w)  \cdot \phi_*( xyz \partial ) \\
			& = \phi_* [\w^{ 2} \partial ,   \mathcal{H}_ w ]_t \\
			& =  [\phi_*  (\w^{ 2} \partial) ,   \phi_*( \mathcal{H}_\w )]_s \\
			& = \frac {s }{3} \F_{\frac{-s}{18}}^{(x',y',z')}(\phi(\w))  x'y'z' \partial_*.
		\end{align*}
		Thus, for any element $ \w $ in $ \mathcal{A}^1(\tilde{E}_6 ; t)$ of weight one, we obtain the equality 
		\begin{equation}\label{eq:phi_xyz2}
			\phi (xyz) =  \frac{ s \F_{\frac{-18}{s}}^{(x',y',z')}( \phi(\w))   }{t \F_{\frac{-18}{t}}^{(x,y,z)}(\w)   } \cdot x'y'z' .
		\end{equation} 
		So the factor 
		\[
			\frac{ s \F_{\frac{-18}{s}}^{(x',y',z')}( \phi(\w))   }{t \F_{\frac{-18}{t}}^{(x,y,z)}(\w)   }
		\]
		is invariant regardless of the selection of $\w$. Consequently, the first assertion is validated.
		
		(2) From the first assertion, any element $\w$ in $ \mathcal{A}^1(\tilde{E}_6 ; t) $ of weight one satisfies 
		\begin{equation} \label{phi_xs}
			   \F_{\frac{-18}{s}}^{(x',y',z')}( \phi(\w))  = \F_{\phi} \cdot \F_{\frac{-18}{t}}^{(x,y,z)}(\w) 
		\end{equation}
		for some constant $ \F_\phi $. 
	From the definition
	$ \F_{\frac{-18}{t}}^{(x,y,z)}(\w)=f_{\frac{-18}{t}} (\alpha ,\beta, \gamma ) $
	and the explicit expansion of the term
	$ \F_{\frac{-18}{s}}^{(x',y',z')}( \phi(\w)) $,
	 we see that the transformation
	 $ \w \mapsto \phi(\w) $ represents a morphism in $ \Grp^{\infty}(\tilde{E}_6)(\frac{-18}{s}, \frac{-18}{t}) $. The equality \eqref{equ:phiw} can be rewritten as 
	 \[ \phi(\w) =\alpha' x'+ \beta' y' +  \gamma' z'
	 \]
	 where 
	 \[ 
		\begin{pmatrix}
			\alpha' \\
			\beta' \\
			\gamma'  
		\end{pmatrix}
		= \begin{pmatrix}
			a_1& b_1& c_1 \\
			a_2& b_2& c_2 \\
			a_3& b_3& c_3  
		\end{pmatrix}
		\begin{pmatrix}
			\alpha \\
			\beta \\
			\gamma  
		\end{pmatrix}.
	 \]
	It means that $ \w \mapsto \phi(\w) $ is represented as the transpose matrix 
	of $  \phi $, and thus the lemma holds.
	\end{proof}

	\begin{proof}[Proof of the second assertion of Lemma \ref{lem:assertion}]
		
	We have known from \eqref{equ:GrpE6} that the second assertion holds for $ k \geqslant 2 $. It suffices to consider the cases $ k = 0 $ and $ k = 1 $.
  	
	We have already shown in (1) of Lemma \ref{lem:groupoid} that  
  \[
	  \Grp^{\infty}(\tilde{E}_6) (t,s) \subseteq \Grp^{0}(\tilde{E}_6) (t,s).
  \]
  It follows from (3) of Lemma \ref{lem:groupoid} that 
  \[
	\Grp^{0}(\tilde{E}_6) (t,s) \subseteq \Grp^{1}(\tilde{E}_6) (t,s).
  \]
  For a subset $ S \subseteq \PGL_3(\mathbb{C}) $, we denote all the transposes of matrices in $ S $ by $ S^{T} $.
 We obtain from Lemma \ref{lem:transpose} that 
 \begin{equation*} 
	\Grp^1(\tilde{E}_6)(t,s)  ^{T} \subseteq \Grp^\infty(\tilde{E}_6)(-\frac{18}{s}, -\frac{18}{t}) .
 \end{equation*} 
	Note that this implies
	\begin{align*}
		\Grp^1(\tilde{E}_6)(t,s) &  \subseteq \Grp^\infty(\tilde{E}_6)(-\frac{18}{s}, -\frac{18}{t}) ^T  \\
		& \subseteq  \Grp^1(\tilde{E}_6)(-\frac{18}{s}, -\frac{18}{t}) ^T \\
		&  \subseteq \Grp^\infty(\tilde{E}_6)(t,s) .
	\end{align*}
	 Finally, we conclude that
	\[ 	  \Grp^0(\tilde{E}_6)(t,s) =  \Grp^1(\tilde{E}_6)(t,s) =\Grp^\infty(\tilde{E}_6)(t,s) . \]
	\end{proof}

   \subsection{The third assertion}
	\begin{lem}
		For $ t,s \in \mathbb{C}(\tilde{E}_6) \setminus \{ 0, 6, 6\rho , 6 \rho^2 \} $, we assume that the morphism $ \phi $ in $ \Grp^1(\tilde{E}_6)(t,s) $ is presented as the matrix \eqref{equ:phi} with inverse \eqref{equ:inverse}.  Then
		\begin{subequations}
			\begin{equation}\label{eq:a1p}
				a_1 ' = ( a_1^2 - \frac{6}{s}a_2 a_3)/ f_{\frac{-18}{s}}(a_1,a_2,a_3),
			\end{equation}
			\begin{equation}\label{eq:b1p}
				b_1 ' = ( a_2^2 - \frac{6}{s}a_1 a_3)/ f_{\frac{-18}{s}}(a_1,a_2,a_3) ,
			\end{equation}
			\begin{equation}\label{eq:c1p}
				c_1 ' = ( a_3^2 - \frac{6}{s}a_1 a_2)/ f_{\frac{-18}{s}}(a_1,a_2,a_3)  .
			\end{equation}
		\end{subequations}
	\end{lem}
	\begin{proof} 
		We maintain the notations of Lemma \ref{lem:transpose}.  
	  If we select $ \w = x \in \mathcal{A}^1(\tilde{E}_6; t)$, then  $ \mathcal{H}_ \w := \w (x\partial_x+y \partial_y + z \partial_z)  $.
		Define $ N(\w):= \langle \w^2 \partial_x, \w^2 \partial_y, \w^2 \partial_z \rangle $ and 
		$ Z = \langle\mathcal{H}_x, \mathcal{H}_y, \mathcal{H}_z \rangle $. Then,
		\[   V(\w): = \{u \in N(\w): [u, Z]_t =0\} = \langle \w^2 \partial_x\rangle .  \]
		This implies that we identify a $1$-dimensional subspace $ D_t'(\w) := \langle \partial_x \rangle$ parameterized by $\w$.
	By computing the Lie bracket of $ \mathcal{H}_\w $ and $L^1(\tilde{E}_6;t) $, we obtain the subspace 
		\[  D_t''(\w) = \langle   \partial_y,   \partial_z\rangle .\]
		Therefore, $ \w $ leads to a decomposition of $ D_t $ as follows:
		\[
		D_t =  D_t'(\w) \oplus D_t'' (\w)= \langle \partial_x \rangle \oplus \langle \partial_y, \partial_z  \rangle. 
		\]
		In other words, we generate the vector space  $ \langle \partial_x \rangle $  through the influence of $ \w $.
		Thus, we achieve a decomposition of  $ x\partial_x+y \partial_y + z \partial_z $ relative to $ \w $, given by  
		\begin{equation} \label{equ:Eular}
			  x\partial_x+y \partial_y + z \partial_z  = \w \partial_x  + ( y \partial_y +  z \partial_z) . 
		\end{equation}
		It yields that $ \partial_x $ (not just the linear space $ \langle \partial_x \rangle$) is determined by $ \w $.

		We turn to study the corresponding decomposition of $D_t$ and image of $ \partial_x $ under $  \phi _* $.
		By definition, we have 
		\[  \phi(\w) = a_1 x' + a_2 y' + a_3 z'. \]
		Then the image of $\mathcal{H}_\w $ equals
		\[  \phi_*(\mathcal{H}_\w)   = (a_1 x' + a_2 y' + a_3 z')({x'}\partial_{x'}+{y'} \partial_{y'}+ z' \partial_{z'}) .\] 
		Evidently,
		\[ \phi_* N(\w) =\langle \phi(\w)^2 \partial_{x'} , \phi(\w)^2 \partial_{y'} , \phi(\w)^2 \partial_{z'}  \rangle. \]
		Since $ \phi_* Z = Z$, it follows that
	  \begin{align*}
	  \phi_* V(\w) & = \{ u  \in \phi_* N(\w)  : [u, Z]_s = 0 \}\\
	  & = \langle  \phi(\w)^{2} \left( ( s  a_1^2 - 6 a_2 a_3) \partial_{x'} +(s  a_2^2 - 6 a_1 a_3) \partial_{y'} +(s  a_3^2 - 6 a_1 a_2) \partial_{z'} \right) \rangle  .  
	  \end{align*}
		Hence, we have 
		\[  \phi_* D_t' (\w)  = \langle ( s  a_1^2 - 6 a_2 a_3) \partial_{x'} +(s  a_2^2 - 6 a_1 a_3) \partial_{y'} +(s  a_3^2 - 6 a_1 a_2) \partial_{z'}\rangle. \]
		Recalling the computation of $ \phi_* D_t'' (\w)  $ as obtained in Equation \eqref{equ:Dtw}: 
		\[   \phi_* D_t'' (\w) = \langle -a_2 \partial_{x'} + a_1 \partial_{y'} ,  - a_3 \partial_{x'} + a_1 \partial_{z'} \rangle , \] 
		we arrive at the decomposition 
		\[
		D_s =  \phi_* D_t' (\w)  \oplus   \phi_* D_t'' (\w)  .
		\]
		Thus, the decomposition of $ x'\partial_{x'}+y' \partial_{y'}+ z'\partial_{z'} $ is given by 
		\begin{align*}
		& \phi_*(  x\partial_x+y \partial_y+ z\partial_z ) \\
		&= x' \partial_{x'}+y' \partial_{y'} +z' \partial_{z'} \\
		&= \frac{\phi(\w)}{f_{\frac{-18}{s}}(a_1,a_2,a_3) } \left(  (  a_1^2 - \frac{6}{s} a_2 a_3) \partial_{x'} +(  a_2^2 - \frac{6}{s} a_1 a_3) \partial_{y'} +(  a_3^2 - \frac{6}{s} a_1 a_2) \partial_{z'} \right) 
		+ \Delta_\w   ,
		\end{align*}
		for some $ \Delta_\w \in \phi_* D_t'' (\w)  $.
		Comparing this with \eqref{equ:Eular}, we see the image of $ \partial_x $ is given by 
		\[  \phi_*(\partial_x) = \frac{1}{f_{\frac{-18}{s}}(a_1,a_2,a_3)} \left(  (  a_1^2 -\frac{6}{s} a_2 a_3) \partial_{x'} +(  a_2^2 - \frac{6}{s} a_1 a_3) \partial_{y'} +(  a_3^2 - \frac{6}{s} a_1 a_2)  \partial_{z'}\right). \]
		Notice that
		\begin{align*}
		\phi_* (\partial_x) &=  a_1' \partial_{x'}+b_1' \partial_{y'}+c_1'  \partial_{z'}.
		\end{align*}
		So we obtain  the equalities \eqref{eq:a1p}\eqref{eq:b1p}\eqref{eq:c1p}.
  \end{proof}

  \begin{lem}\label{lem:a1^3}
  	Assume that $ t,s \in \mathbb{C}(\tilde{E}_6) \setminus \{ 0, 6, 6\rho , 6 \rho^2 \} $. Let $ \phi $ be a homogeneous isomorphism contained in $\Grp^1(\tilde{E}_6)(t,s)$. Denote by $ a_1, a_2, a_3 $ the entries in the first row of the representation matrix \eqref{equ:phi} of $ \phi $. Then only one the following two cases may occur:
  \begin{enumerate}
  	\item  If all elements $ a_1  , a_2  , a_3  $ are nonzero, then $ a_1^3 = a_2^3= a _3^3  $;
  		\item  If at least two elements of $ \{a_1,a_2,a_3\} $ are equal to zero.
  \end{enumerate} 
  \end{lem}

\begin{proof}
 (1) Firstly, we assume that $ a_1,a_2,a_3$ are nonzero. Let $ \phi $ be a morphism contained in $ \Grp^1(\tilde{E}_6)(t,s)$. Lemma \ref{lem:transpose} implies that the transpose matrix of $  \phi^{-1} $ represents a
	morphism in $ \Grp^1(\tilde{E}_6)(\frac{-18}{t}, \frac{-18}{s}) $. Analogous to the equalities \eqref{eq:a1p}-\eqref{eq:c1p}, we deduce the equations
  	  	\begin{subequations}
  		\begin{equation} \label{eq:a1}
  	a_1 = ( a_1^{\prime 2} + \frac{s}{3} b_1'  c_1' )/ f_s(a_1', b_1',c_1'),
  		\end{equation}
	\begin{equation}  \label{eq:a2}
  	a_2 = ( b_1^{\prime 2}  + \frac{s}{3} a_1'  c_1')/ f_s(a_1', b_1',c_1'),
  		\end{equation}
  		\begin{equation}  \label{eq:a3}
  	a_3 = (  c_1^{\prime 2} + \frac{s}{3} a_1'  b_1')/ f_s(a_1', b_1',c_1'). 
  		\end{equation}
  	\end{subequations} 
  	Combining these equations shows that the values $a_i $ with $i =1,2,3$ verify the equality 
  	\begin{align*}
  	     & a_i^3 ( s^2 + 2)   + ( \frac{s^3}{3} + 36) \frac{1}{a_i^3} a_1^2 a_2^2 a_3^2 = \Lambda 
  	\end{align*}
  	where $ \Lambda $ denotes a constant of the form 
	\[
		\Lambda:=2 (a_1^3+a_2^3+a_3^3)+s^2 (f_{\frac{-18}{s}}(a_1,a_2,a_3))^2 f_s(a_1',b_1',c_1').
	\]
  	Equivalently, we find that  
  	$ a_1^3, a_2^3, a_3^3 $ are roots of $G(X) $ defined as 
  	\begin{equation}\label{eq:solution}
  	     G(X):=( s^2 + 2) \cdot X   + ( \frac{s^3}{3} + 36) a_1^2 a_2^2 a_3^2 \cdot \frac{1}{X} - \Lambda.
  	\end{equation} 
  	Since the function $G(X) $ can have at most two distinct roots,  at least two of the values $ a_i^3 $ must coincide.
  	If our lemma does not hold, then we can assume without loss of generality that $ 
  	a_1 ^3= a_2^3 \not= a_3^3  $. From the expression in \eqref{eq:solution}, we observe that both
  	$ ( \frac{s^3}{3} + 36) $ and $( s^2 + 2)  $  are nonzero, which implies that 
  	\[ a_1^3 a_3^3= a_2^3 a_3^3 =   \frac{ s^3 + 108 }{ 3 s^2 + 6  } \cdot a_1^2 a_2^2 a_3^2. \]
  	This can be simplified to the relation
  	\begin{equation}\label{eq:a3_s}
  	a_3  = \frac{ s^3 + 108 }{ 3 s^2 + 6  }  \cdot \frac{a_2^2}{a_1} .
  	\end{equation}

  	We now consider another morphism $ \phi' $ in $ \Grp^1(\tilde{E}_6)(t,s) $ represented by the matrix
  	\[  \begin{pmatrix}
  	a_1 & a_2\rho & a_3  \\
  	b_1 & b_2\rho & b_3  \\
  	c_1 & c_2\rho & c_3  
  	\end{pmatrix} =  \begin{pmatrix}
		a_1 & a_2 & a_3  \\
		b_1 & b_2 & b_3  \\
		c_1 & c_2 & c_3  
		\end{pmatrix}
		\cdot \begin{pmatrix}
			1 &  0  & 0  \\
			 0  & \rho & 0  \\
			 0 & 0 & 1  
			\end{pmatrix}
	\]
  	It is clear that 
  	$ a_1^3 = (a_2 \rho)^3 \not= a_3^3 $. The arguments above applied to $ \phi' $ lead us to the expression
  	\begin{equation}\label{eq:a3_p}
  	a_3  = \frac{ 6 (s \rho)  ^3 + 648 }{ 18 (s \rho) ^2 + 36  } \cdot \frac{(a_2\rho)^2}{a_1}  
  	\end{equation} 
	which parallels \eqref{eq:a3_s}. 
  	Combining Equations (\ref{eq:a3_s}) with (\ref{eq:a3_p}) yields that 
  	\[ \frac{  (s \rho)  ^3 + 108 }{ 3 (s \rho) ^2 + 6  } \cdot \rho^2 = \frac{  s ^3 + 108 }{ 3 s ^2 + 6  } , \] 
  	leading to 
  	\[ \rho^2 = 1  , \]
  	which contradicts the assumption $ \rho^3 = 1 $. So we obtain
  	$ a_1^3 = a_2^3 =a_3^3 $.

  	(2) In the second case, if at least one of the values $ a_i $ equals zero, we can set $ a_3 = 0 $ without loss of generality. The equations \eqref{eq:a1p}-\eqref{eq:c1p} imply:
  	\begin{align*}
  	a_1' &=  a_1^2 / f_{\frac{-18}{s}}(a_1,a_2,a_3),\\
  	b_1' & = a_2^2/  f_{\frac{-18}{s}}(a_1,a_2,a_3),\\
  	c_1' & = -\frac{6}{s} a_1 a_2/ f_{\frac{-18}{s}}(a_1,a_2,a_3).
  	\end{align*}
  	Assuming both $ a_1 $ and $ a_2 $ are nonzero, we obtain $ a_1',b_1', c_1' $ are also nonzero. Applying the result from the first case to the transpose matrix of $ \phi^{-1} $, we get
  	\begin{equation}\label{eq:a1p3}
  	  a_1'^3 = b_1'^3 =c_1'^3  . 
  	\end{equation} 
 Now Equation $  \eqref{eq:a3} $ implies
  	\[  0 = a_3 = ( a_1'^2 + \frac{s}{3} b_1' c_1' )/ f_s(a_1',b_1',c_1'). \]
  	It follows that $ a_1'^2 = - \frac{s}{3} b_1' c_1'$. 
   Combining this with Equations \eqref{eq:a1p3}, we have 
   $ 	s^3= -27$.
   This contradicts our initial assumption about $s$. Therefore, the assumption about $a_1$ and $a_2$ must be false, ensuring that either $a_1$ or $a_2$ is zero.
	In conclusion, the second case of the lemma is validated.
  \end{proof}

 We are in the position to give a complete description for 
 $ \Grp^1(\tilde{E}_6)(t,s)$.
  \begin{proof}[Proof of the third assertion of Lemma \ref{lem:assertion}]
	We are employing the notations of Lemma  \ref{lem:a1^3}.

  	Case (1): There exists some index $  i  = 1,2,3 $ such that $ a_i = 0 $.

  	For instance, if we set $ a_3 = 0 $, applying Lemma \ref{lem:a1^3} reveals that either $ a_1 $ or $ a_2 $ is zero. 
	For the case where $ a_1 \not = 0, a_2 = a_3 = 0$, Equations \eqref{eq:a1p}\eqref{eq:b1p}\eqref{eq:c1p} can be rewritten as
  	\begin{equation*}
  	a_1' = s a_1^2/\theta, \quad
  	b_1' =0, \quad 
  	c_1'= 0.
  	\end{equation*} 
  	Combining these equalities with the original constraints
	\[  b_1 a_1' + b_2 b_1'+ b_3 c_1' = 0 \]
	 and 
	 \[ c_1 a_1' + c_2 b_1'+ c_3 c_1' = 0 \]
	 derided from the definition, we have
  	$  b_1 =c_1 = 0  $.
  	Upon extending the above arguments to  $ \{b_1, b_2,b_3\}  $ and $ \{c_1, c_2,c_3 \} $ respectively, we conclude that
  	$  \phi $ is a matrix of types \Rmnum{1} or \Rmnum{2}, where 
  	\[ \text{\Rmnum{1}}= \begin{pmatrix}
  	\lambda_1 & 0  & 0  \\
  	0 & \lambda_2  & 0 \\
  	0 & 0  & \lambda_3   
  	\end{pmatrix} 
  	\text{ and \Rmnum{2}}
  	= \begin{pmatrix}
  	\lambda_1  & 0  & 0  \\
  	0 & 0 & \lambda_2 \\
  	0 & \lambda_3 & 0   
  	\end{pmatrix} .
  	\]
  	If we also consider the cases $ a_3 \not=0 , a_1 = a_2 = 0 $ and $a_2 \not= 0,  a_1 = a_3 = 0 $, then we obtain four additional types:
  	\[  
	\text{\Rmnum{3}} =\begin{pmatrix}
  	0 &  \lambda_1  & 0  \\
  	\lambda_2 &  0 & 0 \\
  	0 & 0  & \lambda_3   
  	\end{pmatrix},
  	\text{\Rmnum{4}}
  	= \begin{pmatrix}
  	0 &  \lambda_1  & 0  \\
  	0 & 0 & \lambda_2 \\
  	\lambda_3 & 0 & 0   
  	\end{pmatrix} ,  
	  \text{\Rmnum{5}}= \begin{pmatrix}
  	0 &    0& \lambda_1  \\
  	\lambda_2 &  0 & 0 \\
  	0 & \lambda_3  & 0   
  	\end{pmatrix} 
  	,
  	\text{\Rmnum{6}}=\begin{pmatrix}
  	0 &   0 & \lambda_1  \\
  	0 &  \lambda_2  & 0 \\
  	\lambda_3 &0 & 0   
  	\end{pmatrix} .
  	\]
  	The first assertion of Lemma \ref{lem:transpose} states that 
	\[
			\F_\phi := \frac{  \F_{\frac{-18}{s}}^{(x',y',z')}( \phi(\w))   }{ \F_{\frac{-18}{t}}^{(x,y,z)}(\w)   }
	\]
	is invariant for any $ \w $ of weight one.
    By choosing $ \w = x, y, z $ respectively, we get the equality
  	\begin{equation}\label{equ:f18}
		 f_{\frac{-18}{s}}(a_1,a_2,a_3) = f_{\frac{-18}{s}}(b_1,b_2,b_3)=f_{\frac{-18}{s}}(c_1,c_2,c_3).
\end{equation}
  	 Substituting the values of $ a_i,b_i,c_i $ from the six types \Rmnum{1}-\Rmnum{6}, we conclude that 
	  \[\lambda_1^3 = \lambda_2^3 =\lambda_3^3  
	  \]
	in each type respectively.
	There are exactly $ 54 $ matrices contained in these six types up to scalar multiplication and each matrix is generated by $ A_1 ,A_3, A_4 $.

  	Case (2): Suppose that $a_1,a_2,a_3$ are nonzero.

  	From Case (1), we know that all entries of $ \phi $ are nonzero.
  	Applying Lemma \ref{lem:a1^3} to the sets $\{ a_1, a_2, a_3 \}$, $\{ a_1, b_1, c_1\}$,  $ \{ a_2,b_2,c_2 \} $ and $ \{ a_3,b_3,c_3 \} $ respectively, we conclude that
  	all entries of $ \phi $ are contained in the set $  \{ a_1  ,a_1 \rho, a_1\rho^2 \} $.
  	
  	From Equation (\ref{equ:f18}), we get
  	\begin{equation}\label{eq:aaa}
  	a_1 a_2 a_3  = b_1 b_2 b_3 = c_1 c_2 c_3.
  	\end{equation}
  	By considering the transpose of $ \phi $, we deduced from Equation \eqref{eq:aaa} that 
  	\begin{equation}\label{eq:abc}
  	a_1 b_1 c_1  = a_2 b_2 c_2 = a_3 b_3 c_3  . 
  	\end{equation}
  	Fix $ a_1 \in \mathbb{C}^* $. Computational checks reveal that only $ 162\times3 $ non-singular matrices verify Equations \eqref{eq:aaa} and \eqref{eq:abc}.
  	Up to scalar multiplication, there are exactly $ 162 $ matrices. One may check directly that each matrix is generated by $ A_1 ,A_2,A_3, A_4 $.
  \end{proof}

\subsection{Continue family of representations of the first Yau algebra}\label{sec:representation}
In this section, we prove assertion (3) of Theorem \ref{thmA}. Following the approach in \cite{Yau1983Continuous}, we construct a basis in which the coefficients of the Lie bracket are independent of $t$. For the entirety of this section, we assume that $  t^3 \not = 0 , 216, -27 $.

Recalling that the basis of $ \mathcal{A}^{1}(\tilde{E}_6; t )$ is given by 
 \[1, x, y, z, x^2, y^2, z^2, x y, x z, y z,  xyz. \]
We choose an alternative basis for $ L^{1}(\tilde{E}_6; t ) $ as follows. 
\begin{enumerate}
	\item Degree 0:\\
	$  e_1:= x\partial_x  + y\partial_y + z\partial_z $;
\item Degree 1:\\
$  e_2:= 3 x^2 + t y z \partial_x  $,
$  e_3:= 3 y^2+ tx z \partial_x $,
$  e_4:= 3  z^2+tx y  \partial_x  $,
$  e_5:=  3  x^2 + t y z  \partial_y $, 
$  e_6:= 3 y^2 +t x z  \partial_y $,
$ e_7:= 3 z^2 + t x y \partial_y$,
 $ e_8:=  3 x^2 +   t y z\partial_z$,
$  e_9:=   3 y^2+t x z \partial_z $,
$  e_{10}:= 3  z^2  + t x y \partial_z $,
$  e_{11}:=x^2\partial_x+ 2 x y\partial_y  $,
$  e_{12}:=x^2 \partial_x +  2 x z\partial_x  $,\\
$  e_{13}:=2 x y\partial_x+  y^2 \partial_y  $,
$  e_{14}:=  y^2\partial_y+ 2 y z\partial_z $,
$  e_{15}:=  2 y z\partial_y+ z^2\partial_z $,\\
$  e_{16}:=2 x z\partial_x+  z^2\partial_z  $,
$  e_{17}:= (1/t)x^2\partial_x   $,
$  e_{18}:=  (1/t)y^2 \partial_y   $,
$  e_{19}:=  (1/t)z^2\partial_z  $;
\item Degree 2:\\
$  e_{20}:=   xyz \partial_x $,
$  e_{21}:=   xyz \partial_y  $,
$  e_{22}:=   xyz\partial_z $.
\end{enumerate}

Observe that  $ L^{1}(\tilde{E}_6; t ) $ is a graded Lie algebra with the decomposition:
 \[L^{1}_t = D_0 \oplus D_1   \oplus D_2, \]
 where $ D_i $'s are generated by homogeneous derivations of degree $ i $ respectively. Since there are no derivations of degree three, we have
 $ [D_1,D_2] = 0 $. The subspace $ D_0 $ is generated by the Euler derivation $ e_1 $. We obtain
 $ [ e_1, u ] = u $ for $ u \in  D_1 $  and
 $ [ e_1, u ] = 2 u $ for $ u \in  D_2 $.    Next, we define the subspaces $ U:=\langle  e_2, \cdots, e_{10} \rangle $, $ V:=\langle  e_{11}, \cdots, e_{16} \rangle $, and 
 $W:=\langle  e_{17}, e_{18}, e_{19} \rangle $. This leads to the decomposition:
  \[ D_1 = U \oplus V \oplus W  . \] 
 One can check that
 \[   [U,U] =0  , [U,V] =0 ,   [W,W]=0  ,   [V,W] =0  . \] 
 Furthermore, the matrix representation for Lie bracket $ [ U,W] $ is given by
 \[\left( \begin{matrix}
 2 e_{20}& 0& 0&  2   e_{21}& 0&0 & 2 e_{22}& 0& 0\\
 0 & 2 e_{20}& 0& 0&   2  e_{21}& 0&0 & 2 e_{22}& 0 \\
 0& 0 & 2 e_{20}& 0& 0&   2 e_{21}& 0&0 &2 e_{22} 
 \end{matrix}  \right),
 \]
and the representation for $ [V,V] $ is given by
 \[\left( \begin{matrix}
 0 &            0&          0& 4 e_{22}&           0&   -4 e_{21}  \\  
 0 &            0&  -4 e_{22}&          0&   4 e_{21} &            0 \\  
 0 &  4 e_{22}&          0&          0&  -4 e_{20} &            0 \\  
-4 e_{22}&     0&        0&         0 &          0 &   4 e_{20}  \\ 
 0&     -4 e_{21}&  4 e_{20}&         0 &          0 &   0  \\ 
 4 e_{21}&     0&        0&      -4 e_{20} &          0 &   0   
 \end{matrix} \right).
 \] 
 This implies that all coefficients of Lie bracket under our basis are constant, indicating that the first Yau algebra first Yau algebra does not depend on the parameter $ t $. 
 We have finished the proof of Theorem \ref{thmA}.

 \section{Groupoid of the $\tilde{E}_7$-family}
 \subsection{Outline of the proof of Theorem \ref{thmB}}
Recall that the $ \tilde{E}_7 $-family is defined as the zero locus of  polynomial
\[
f_t = x^4 + y^4 + z^2 + t x^2 y ^2 , \text{where $ t ^2 \not= 4 $}.
\]
Firstly, we can directly verify that each matrix $ \tilde{B}_{\alpha,\beta, \gamma} $ in Theorem \ref{thmB} induces an isomorphism from $ \mathcal{A}^{k}(\tilde{E}_7;t)$ to $ \mathcal{A}^{k}(\tilde{E}_7;s) ) $, where $s=\pi'(B_{\alpha,\beta})(t) $ and $k \in \mathbb{Z}_{\geqslant 0 }\cup \{ \infty \}$. For example, let $ \phi $ be the morphism represented as $ \tilde{B}_{\alpha,\beta, \gamma} $ with $ \gamma = \sqrt{2 + t \alpha^2 }$. In terms of transformations, we can express this as:
\[
\begin{pmatrix}
	x \\
	y \\
	z \\
\end{pmatrix}	=
\begin{pmatrix}
	1 & \beta & 0  \\
	\alpha & -\alpha \beta & 0 \\
	0 & 0 & \gamma
\end{pmatrix}	
\begin{pmatrix}
	x' \\
	y' \\
	z' 
\end{pmatrix}.
\]
Restricting to the coordinates $(x,y)$, we get 
\[ h = \pi' (B_{\alpha, \beta}) = \left[t \mapsto \frac{\beta^2 (12 -2 \alpha^2 t )}{2 + t \alpha^2 } \right]\]
by the definition of $\pi'$ in \eqref{equ:pip}. A direct computation shows 
\begin{align*}
	f_t(x,y,z) &= (x' + \beta y')^4 + (x' - \beta y') ^4 + \gamma^2 z'^2 + \alpha^2 t (x'^2  - \beta^2 y'^2) ^2\\
	& = 2x'^4 + 2y'^4 + 12 \beta^2 x'^2 y'^2 + \gamma^2 z'^2 + \alpha^2 t (x'^4  - 2 x'^2 \beta^2 y'^2 + y'^4 )\\
	&= (2 + \alpha^2 t) f_s(x',y',z')
\end{align*}
where 
\[
s = \frac{\beta^2 (12 -2 \alpha^2 t )}{2 + t \alpha^2 } = h(t)	. \]
It follows that $ \tilde{B}_{\alpha, \beta, \gamma}$ is a morphism in groupoid $\Grp^{\infty}(\tilde{E}_7)$.

From \cite{Seeley1990Variation}, we know that for $t \in \{ 0, \pm 6 \} $, we encounter the jump point case, in which the zeroth Yau algebra  $  L^0(\tilde{E}_{7}; t ) $ is $12$-dimensional. In contrast, for the cases where $t \not = 0, \pm 6 $, $  L^0(\tilde{E}_{7}; t ) $ is $11$-dimensional. 

The proof for  $t = 0 , \pm 6 $ of Theorem \ref{thmB} follows the same argument as in the assertion (1) of Lemma \ref{lem:assertion}. Therefore, we omit the details.

When $k \geqslant 2$, we conclude from Lemma \ref{lem:groupoid} that 
\[
	\Grp^k(\tilde{E}_7) = \Grp^\infty(\tilde{E}_7) \subseteq \Grp^i (\tilde{E}_7) 
\]
for $ i = 0 , 1$. So the assertion (2) in Theorem \ref{thmB} is concluded by the assertion (1). 

It remains to prove the case $ k = 0, 1 $ and $ t \not = 0, \pm 6$. We will compute the groupoid $\Grp^0(\tilde{E}_7)$ by restricting to the coordinates $(x,y) $. In fact, the complete collection of isomorphism with respect to $(x,y)$ are exactly the group $G'$, which is also described in \cite{Eastwood2004Moduli} without proof. We utilize the Torelli-type theorem from \cite{Seeley1990Variation} to show that the isomorphisms of the $\tilde{E}_7$-family preserve the four lines in the zeroth Yau algebra. In this sense, the group $G'$ is indeed the symmetric group acting on the four lines. 

For the groupoid $\Grp^1(\tilde{E}_7)$, we need to prove the assertion (3) in Theorem \ref{thmB}, namely the Torelli-type theorem of the first Yau algebra, which implies that the morphism of groupoid $\Grp^1(\tilde{E}_7)$ are derived from $\Grp^\infty(\tilde{E}_7)$. As a consequence of the Torelli-type theorem, we can reduce the groupoid $\Grp^1(\tilde{E}_7)$ to $\Grp^\infty(\tilde{E}_7)$, which completes the proof of Theorem \ref{thmB}.
\subsection{The zeroth Yau algebra}
We assume that $t^2 \not= 0 ,4,36$. The moduli algebra is written as
\begin{align*}
	\mathcal{A}^0(\tilde{E}_7;t ) = \mathbb{C}[[x,y]] / I_0(t)
\end{align*}
where, as usual,
 \[ 
I_0(t) = \left \langle  4x^3 + 2t x y^2, 4 y^3 + 2 t x^2 y ,z \right \rangle.
\]
Notice that for $ t \in \mathbb{C}(\tilde{E}_7)$, we have 
\[
	(4- t^2) x^3 y =  y ( 4x^3  + 2t x y^2 ) - \frac{t}{2}x (4 y^3 + 2 t x^2 y) =0 \mod I_0(t)  
 \]
 and similarly 
 \[
	(4- t^2) y^3 x =  x ( 4y^3  + 2t y x^2 ) - \frac{t}{2}y (4 x^3 + 2 t y^2 x) =0 \mod I_0(t).  
 \]
The monomials 
\[
	1, x, y, x^2, xy , y^2, x^2 y , y^2 x , x^2 y^2 
\]
form a basis of $\mathcal{A}^0(\tilde{E}_{7}; t ) $
with multiplication rule 
\[
x^3  =  - \frac{t}{2}  x y^2 ,\quad  y^3 = - \frac{t}{2} x^2 y , \quad  x^3 y = y^3 x = 0.	
\]
From \cite{Seeley1990Variation}, we know that the jump points of the $\tilde{E}_7 $-family are $ t = 0, \pm 6$, in which case 
the zeroth Yau algebra $  L^0(\tilde{E}_{7}; t ) $ is $12$-dimensional. In contrast, for the case $t \not = 0, \pm 6 $, the Lie algebra $  L^0(\tilde{E}_{7}; t ) $ is 11-dimensional and 
there exists a basis for $L^0(\tilde{E}_{7}; t ) $: 
\begin{itemize}
	\item Degree 0:\\
	$ e_0 = x \partial_x + y \partial_y$;
	\item Degree 1:\\
	$e_1 = x^2 \partial_x + xy \partial_y$, $e_2 = yx \partial_x + y^2 \partial_y $,
	$ e_3 = (t^2 - 12) xy \partial_x + 4 tx^2 \partial_y $, \\
	$ e_4 = 4t y^2 \partial_x +(t^2 -12) xy \partial_y$;
	\item Degree 2:\\
	$e_5 = x^2 y  \partial_x $, 
	$e_6 = xy^2 \partial_y $, 
	$e_7 = xy^2 \partial_x $,
	$ e_8 = x^2 y \partial_y$;
	\item Degree 3:\\
	$e_9 =  x^2 y^2 \partial_x $,
	$e_{10} = x^2 y^2 \partial_y $.
\end{itemize}
Choose a solution $C_t$ of the equation 
\begin{equation}\label{equ:C_t}
	C_t^2 + \frac{1}{C_t^2} = \frac{12}{t}.
\end{equation}
We define $l_1(t),\ldots,l_4(t)$ (without orderings) to be the four lines on the plane $ \langle e_9, e_{10} \rangle $  whose slopes are exactly the solutions of equation \eqref{equ:C_t}, specifically:
\begin{align*}
	l_1(t) =& \left \langle C_t e_9 +  e_{10} \right \rangle, \\
	l_2(t) =& \left \langle  e_9 + C_t e_{10}\right \rangle,\\
	l_3(t) =& \left \langle -C_t e_9 +  e_{10}\right \rangle,\\
	l_4(t) =& \left \langle  e_9  -C_t e_{10}\right \rangle.
	\end{align*}
 The set $L^4$ of these four lines is shown to be invariant under homogeneous isomorphisms. 
These lines have a cross-ratio, relative to a choice of orderings. Considering the symmetric group $ S^4 $ acting on the four lines, the $24$ orderings provide six different values of the cross-ratio:
\[
 \frac{6+t}{2t}, \frac{2t}{6+t }, \frac{6+t}{6-t }, \frac{6-t}{6+t }, \frac{2t}{t-6 },\frac{t-6}{2t}
\]
which constitute the six continuous invariants of $  L^0(\tilde{E}_{7}; t ) $. The set of the six continuous invariants is equivalent to the $j$-function $ j(\tilde{E}_7;t)$.
 In addition, the transformations preserve $L^4$ and leaves the six continuous invariants unchanged generate the Klein four-group.
We adopt the approach that the morphisms in $ \Grp^0(\tilde{E}_7) $ induce linear maps on the plane $ \langle e_9,e_{10} \rangle $. In terms of matrices, the morphisms in $ \Grp^0(\tilde{E}_7) $ essentially derive the action of $ G' $ on the four lines when restricted to the coordinates $x $ and $ y $.
\begin{lem}\label{lem:symmetric}
	The group $ G' $ acts on the set $ L^4 = \{l_1(t), \ldots , l_4(t) \}$ in one-to-one correspondence with the action of symmetric group $S^4$. Specifically, we have:
	\[
		\phi_* (l_i(t)) = l_{\sigma(i)}( s )
	\]
	for $\sigma \in S^4$, $ \phi \in \Grp^{0}(t,s) \cap G' $ and $t,s $ contained in some domain of $ \mathbb{C}(\tilde{E}_7) $. In particular, the subgroup $ G'_0 \subseteq G' $ acts on $ \{l_1(t), \ldots , l_4(t) \}$ as the Klein four-group.
\end{lem}
\begin{proof}
	We only check for the isomorphism $ \phi $ expressed as the matrix $ B_{\alpha, \beta } $ acting on $l_1$. Notice that 
	\[
		B_{\alpha,\beta}^{-1} =  
		\begin{pmatrix}
			\frac{1}{2} &  \frac{1}{2 \alpha }  \\
			\frac{1}{2  \beta} & -\frac{1}{2 \alpha \beta} 
		\end{pmatrix}.
	\]
	Following the notation in \eqref{equ:partial} we have 
	\[
	\phi_*(e_9) = \phi(x^2 y^2) \phi_*(\partial_x) = (x'+\beta y')^2 (\alpha x' - \alpha\beta y' )^2 (\frac{1}{2} \partial_{x'} + \frac{1}{2 \beta} \partial_{y'})
	\]
	and 
	\[
	\phi_*(e_{10}) = \phi(x^2 y^2) \phi_*(\partial_y) = (x'+\beta y')^2 (\alpha x' - \alpha\beta y' )^2 (\frac{1}{2 \alpha} \partial_{x'} - \frac{1}{2 \alpha \beta} \partial_{y'}).
	\]
	Then the line $ \phi_* (l_1(t)) $ has a direction vector 
	\[
		\phi_*( C_t e_9 +  e_{10}) = \lambda ( \alpha \beta C_t + \beta  ) e_{9} + \lambda (\alpha C_t - 1)  e_{10} ,
	\]
	for some constant $ \lambda $. 
	So the slope of $ \phi_* (l_1(t)) $ is given by
	\[
		\eta = \frac{\alpha \beta C_t + \beta  }{\alpha C_t - 1 }.
	\]
	Therefore,  we obtain 
	\begin{align*}
		\eta^2 + \frac{ 1}{\eta^2 } & = \frac{ (\alpha \beta C_t + \beta )^2 }{ ( \alpha C_t - 1 )^2 } + \frac{ ( \alpha C_t - 1 )^2 }{ (\alpha \beta C_t + \beta )^2 } \\
		& = \frac{ (\alpha \beta C_t + \beta )^4 +( \alpha C_t - 1 )^4}{
			(\alpha \beta C_t + \beta )^2 ( \alpha C_t - 1 )^2
		} \\
		& = \frac{2 C_t^4 + 12 \alpha^2 C_t^2 + 2 }{
			\beta^2 (C^4_t - 2 \alpha ^2 C_t^2 + 1 )
		} \\
		& =  \frac{ 2 \frac{12}{t} + 12 \alpha^2 }{ \beta^2 \frac{12}{t} - 2 \alpha^2 \beta^2 }=  \frac{  12 + 6 t \alpha^2 }{ 6\beta^2 -  \alpha^2 t \beta^2 }  .
	\end{align*}
	Since $s = \frac{\beta^2(12 - 2 \alpha^2 t )}{2 + t \alpha^2 }$, we can express
	\[
		\frac{12}{s} =  12 \cdot \frac{2 + t \alpha^2 }{\beta^2(12 - 2 \alpha^2 t )} = \eta^2 + \frac{1}{\eta^2},
	\]
	which coincides with Equation \eqref{equ:C_t}.
	It yields that $ \eta $ is one of the conjugate elements of $ C_{s} $, and therefore the image of $ l_1(t)  $ is just one of the lines $ l_1(s), \ldots, l_4(s) $. 
\end{proof}
Now we are in the position to prove  the case $k=0$ in the assertion (1) of Theorem \ref{thmB}.  
\begin{proof}[Partial proof of Theorem \ref{thmB}]
It suffices to show that the restriction of the morphism of $ \Grp^0(\tilde{E}_7)$ to $(x,y)$ are contained in the group $ G'$.
Given a morphism $ \phi $ of $ \Grp^0(\tilde{E}_7)$, we know that $ \phi $ induces a morphism $ \phi_* $ on the zero Yau algebra, which preserves the set $L^4$ of four lines (not necessary in order). By composing with some matrix in $G'$ if necessary, we can assume the $ \phi_* $ preserves the set $L^4$ and the $ j$-function. We have argued that only the transformations arising from the Klein four-group leaves the $ j$-function invariants unchanged. According to Lemma \ref{lem:symmetric}, the  Klein four-group are presented by $G_0' $ with respect to coordinates $(x,y)$. 
Thus, we conclude that the transformations on $(x,y)$ preserving the set of the four lines generate the group $ G'$, which establishes the case $k=0$ of Theorem \ref{thmB}.
\end{proof}

\subsection{The first Yau algebra of the $\tilde{E}_{7}$-family}
The first moduli algebra of the $\tilde{E}_{7}$-family is given by 
\[
	\mathcal{A}^1(\tilde{E}_7; t) = \mathbb{C} [[x,y,z]]/ I_1(t)
\]
where 
\[ 
I_1(t) = \left \langle  2 x^4 + t x^2y^2, 2 y^4 + tx^2y^2, xy^3, yx^3 , xz ,y z ,z^2  \right \rangle.
\]
For $ t \not\in \{ 0, \pm 6 \} $,
we construct the $ \mathbb{C}$-linear basis of $ L_1(t) $ given by $ \E_1 , \ldots, \E_{23}$:
\begin{itemize}
	\item degree 0: \\
	$ \E_1 = z \partial_z $ and $ \E_2 = x \partial_x + y \partial_y $;
	\item degree 1: \\
	$ \E_3 = x^2 \partial_x $, $ \E_4 = y^2 \partial_x $, $\E_5 = xy \partial_x $, $ \E_6 = x^2 \partial_y $, $ \E_7 = y^2 \partial_y $, $ \E_8 = xy \partial_y $, $ \E_9 = z \partial_x$,$ \E_{10} = z \partial_y$, 
	$ \E_{11}= (2 x^3  + t x y^2 )\partial_z $, $ \E_{12}= (2 y^3  + t x^2 y )\partial_z $;
	\item degree 2: \\
	$ \E_{13} = x^3 \partial_x $, $ \E_{14} = x^2 y \partial_x $, $ \E_{15} = x y^2 \partial_x $, $ \E_{16} = y^3 \partial_x $, \\
	$ \E_{17} = x^3 \partial_y $, 
	$ \E_{18} = x^2 y \partial_y $, $ \E_{19} = x y^2 \partial_y $, $ \E_{20} = y^3 \partial_y $,  $\E_{21} = x^2 y^2 \partial_z $;
	\item degree 3: \\
	$ \E_{22} = x^2 y^2 \partial_{x} $ and $ \E_{23} = x^2 y^2 \partial_y $. 
\end{itemize}

We denote by $ [-,-] $ the Lie bracket of $ L^1(\tilde{E}_7; t) $ and use the notations $[U,V]$ and $ [v,U] $ as in Notation \ref{Nt:UV}.  We aim to construct four distinct lines in a canonical order from the Lie structure of $ L^1(\tilde{E}_7; t) $. 

\begin{lem}\label{lem:fourlines}
	There exists a sequence $(H_1,H_2,H_3,H_4)$ of four lines in $ L^1(\tilde{E}_7; t) $ which is invariant up to homogeneous isomorphisms. 
\end{lem}
\begin{proof}

Denote by $ D^i $ the weighted homogeneous part of $ L^{1}(\tilde{E}_7 ; t) $ of degree $ i $. The existence of the four lines can be verified  through the following construction procedure:
\begin{enumerate}
	\item The subspace $ \langle \E_{1} \rangle $ of $D^0$. 

\item The decomposition of $D^1$: 
\begin{equation}\label{equ:D1}
	D^1 = D^1_0 \oplus D^1_{1} \oplus D^{1}_{-1} 
\end{equation}
where $ D_0^{1} = \langle \E_{3} , \ldots, \E_{8} \rangle $, and $ D_{1}^1 = \langle  \E_{9}, \E_{10} \rangle $ and $ D_{-1}^{1} = \langle \E_{11}, \E_{12}\rangle $.

\item  The $4$-dimension subspace $ V $ of $ D_{0}^{1} $ generated by 
\begin{equation}\label{equ:Z}
	z_1 = x \E_2 = \E_3 + \E_8 , \quad  z_2 = y \E_2 = \E_5 + \E_7,
\end{equation}
and 
\begin{equation}\label{equ:V}
	v_1 = 4t \E_4 +8 \E_3 + (t^2-4) \E_8, \quad  v_2 = 4t \E_6 +8 \E_7 + (t^2-4) \E_5.  
\end{equation}

\item  The $2$-dimensional subspace $ Z $ of $ V $ generated by $ z_1 $ and $ z_2 $ with property $[Z,Z] = 0$. 

\item  The line $ H_1 $ of $[V,V] \subseteq D^{2} $ with director vector 
\begin{align*}
	h_1 &= (tx^2 + 2y^2) y \partial_x - (ty^2 +2x^2) x \partial_y\\
	&=t \E_{14} + 2 \E_{16} - 2 \E_{17} - t\E_{19} .
\end{align*}
\item  The $4$-dimensional subspace $U $ containing $Z$ with basis $ z_1, z_2 $ and 
\begin{equation}\label{equ:U}
	u_1 = \frac{t}{3} \E_3+ 2 \E_4  - \frac{2t}{3} \E_8, \quad  u_2 = \frac{t}{3} \E_7 + 2 \E_6 - \frac{2t}{3} \E_5 ,
\end{equation}
which verifies the property $ Z = U \cap V $. 
\item The line $H_2$, being the intersection of  $ [U,U] $ and $ [ V,Z ] $. 



\item  The complementary space $V^*$ of $Z$ in $V $, and $ H_3 = [V^* ,V^* ] $.

\item The complementary space $U^*$ of $Z$ in $U $, and $ H_4 := [V,U^*] \cap \langle H_1,H_2 \rangle $.  
\end{enumerate}

The detailed constructs are given as follows:

(1) By restricting the Lie bracket to 
\[
	[-,-] : D^0 \times D^3 \to D^3 ,  
\]
 we find that the vectors $ v \in D^{0}$ subject to the property 
\[
	[v , D^3] = 0 
\]
are contained in $ \langle \E_1 \rangle  $.
Therefore, $ \langle \E_1 \rangle $ is an invariant linear subspace.

(2) The vector $\E_{1}$ induces the linear map:
\[
	[\E_1, -] : D^1 \to D^1
\]
with eigenvalues $ 0, 1, -1 $. Hence, there exists the Jordan decomposition \eqref{equ:D1}, where the eigenspace $D_0^1,D_1^1,D_{-1}^1 $ correspond to the eigenvalues $ 0, 1, -1 $.

(3) The subspace $ V \subseteq D^1_0 $ is defined as the collection of vectors $ v\in D^1_0 $ such that
\[
	[v, D^{1}_{-1}] = 0 . 
\]  
We can verify that $ V $ is spanned by $v_1, v_2, z_1, z_2 $ in Equations \eqref{equ:V} and \eqref{equ:Z}.

(4) We define $ Z $ to be the collection of $ x \in V$ such that there exists some $ y \in V $, linearly independent on $x $, subject to 
\[
	[x,y] = 0 . 
\]
Equivalently, for such $x$, the kernel of $ [x,-] : V \to V $ is at least $2$-dimensional.  It is clearly that $ Z $ is invariant up to isomorphism. 
Notice that $ [z_1, z_2] = 0 $, so $ Z $ contains the linear subspace generated by $z_1, z_2 $.  The equality $Z = \langle z_1, z_2 \rangle $ can be checked by observing that the two bilinear maps
\[
	[-,-] :V/ Z \times Z \to D^2 
\]
and 
\[
	[-,-] : V/ Z \times V/ Z \to D^2/ [V,Z] 
\]
are non-degenerate.

(5) We define $ H_1 $ as the collection of $ x \in [V,V] $ such that 
$[x, V] = 0 $. 
A direct computation shows that $ H_1 $ is  $1$-dimensional with basis $ h_1$.

(6) It can be checked that 
\[
	h_1 = [u_2,z_1] - [u_1, z_2] = (t x^2y + 2y^3) \partial_x - (t y^2 x + 2 x^3)\partial_y
\]
where $ u_1 , u_2 \in D^1 $ are defined in \eqref{equ:U}.
Therefore, the space $ U $ spanned by $ u_1, u_2, z_1, z_2 $ is the minimal $4$-dimensional space such that 
\[
	H_1 \subseteq [U,Z]  \text{ and } Z\subseteq U.
\]
By construction, $ U $ is also invariant.

(7) It is straightforward to show that 
\[
	H_2 : =  [U,U] \cap [V,Z]
\]
is $1$-dimensional with basis 
\begin{align*}
	h_2 & = [v_1, z_2] - [v_2 ,z_1] = \frac{3}{2}[u_1,u_2] \\
	& = ( (t^2 -12 ) x^2y -4t y^3 )\partial_x - ((t^2 -12 ) y^2x - 4t x^3) \partial_y \\
	& = (t^2-12) \E_{15} -4t \E_{16} + 4 t \E_{17} - (t^2-12) \E_{18}.
 \end{align*}
	
(8) The space $ V^* = \{ v \in V| [v, H_2] =0  \}$ is an invariant subspace of $ V $ with basis $v_1, v_2 $ as defined in \eqref{equ:V}. 
Then the linear subspace $ H_3 := [V^*, V^*] $ of $D^2$ is spanned by 
\[
h_3 : = [v_1, v_2] = \frac{4 t (t^2 -36)  }{3} h_1 - \frac{t^2 + 12 }{3} h_2  .
 \]
(9) Similarly, the space $ U^* = \{ u \in U | [u, H_3] =0  \}$ is an invariant subspace of $ V $. The intersection of $ [V, U^*] $ and $ \langle H_1, H_2 \rangle $ is spanned by the vector
\[
	h_4 = 5 (t^2 + 12)^2 h_1 + 4 t (t^2 -36) h_2. 
\]
The line $ H_4$ spanned by $ h_4 $ is obviously invariant.
\end{proof}
With the aid of these four invariant lines, we can establish the Torelli-type theorem for $ L^1(\tilde{E}_7;t) $.
\begin{proof}[Proof of part (3) in Theorem \ref{thmB}]
	From Lemma \ref{lem:fourlines}, the sequence of lines $(H_1,H_2,H_3,H_4)$ is invariant under homogeneous isomorphisms. Moreover, the cross ratio of $ (H_1,H_2,H_3,H_4) $ is given by 
\[
	-\frac{5}{16} \frac{(t^2 +12)^3}{t^2(t^2-36)^2} = -\frac{5}{16} \frac{(t^2 +12)^3}{(t^2 +12)^3- 108 (t^2-4)^2} = \frac{5}{16} \frac{j(\tilde{E}_7)}{1- j(\tilde{E}_7)},
 \]
 which aligns with the $j$-function $j(\tilde{E}_7)$. Therefore, the first Yau algebra determines the isomorphism classes of the $\tilde{E}_7$-family.
\end{proof}
Finally, we would like to utilize the invariance of $H_1$ to compute the groupoid $ \Grp^1(\tilde{E}_7)$ for the case $k = 1 $ of the assertion (1) in Theorem \ref{thmB}. 
 \begin{proof}[Complete the proof of Theorem \ref{thmB}]
	It suffices to demonstrate that, when restricting to the coordinates $(x,y)$, the groupoid $\Grp^{\infty}(\tilde{E}_7)(t,s)$ coincides with $\Grp^{1}(\tilde{E}_7)(t,s) $. Specifically, the restriction of a morphism of $\Grp^{1}(\tilde{E}_7)(t,s)$ to $ (x,y) $ induces a homogeneous homomorphism of $ x^4 + y^4 + tx^2y^2$.   
 Notice that this, together with Lemma \ref{lem:groupoid}, implies that 
 $ \Grp^{k}(\tilde{E}_7)(t,s) $ are the same up to transformations along the $ z $-axis.

 Let $\phi$ be a morphism in $\Grp^{1}(\tilde{E}_7)(t,s)$, represented as the matrix
 \[ \begin{pmatrix}
	x   \\
	y   
 \end{pmatrix}  = \begin{pmatrix}
	a & b \\
	c & d
 \end{pmatrix}
 \begin{pmatrix}
	x'  \\
	y'  
 \end{pmatrix}
 \]
 in the coordinates $(x,y)$.
 Let $\Delta = ab - bc $ be the determinant. Then the inverse of $ \phi$ is given by 
 \[ \begin{pmatrix}
	x'   \\
	y'  
 \end{pmatrix}  = \frac{1}{\Delta} \begin{pmatrix}
	d & -b \\
	-c & a
 \end{pmatrix}
 \begin{pmatrix}
	x  \\
	y  
 \end{pmatrix}.
 \]
 Then, the partial derivatives can be expressed as
 \[
	\partial_{x} = \frac{d}{\Delta} \partial_{x'} - \frac{c}{\Delta} \partial_{y'}
 \]
 and
 \[
	\partial_{y} = -\frac{b}{\Delta} \partial_{x'} + \frac{a}{\Delta} \partial_{y'}.
 \]
 The image of the director vector  $ h_1 $ of $ H_1 $ under $ \phi_* $ is given by 
 \begin{align*}
	\phi_*(h_1) & =   (tx^2 y +2y^3)  (\frac{d}{\Delta} \partial_{x'} - \frac{c}{\Delta} \partial_{y'}) - (ty^2 x +2x^3)  (-\frac{b}{\Delta} \partial_{x'} + \frac{a}{\Delta} \partial_{y'}) \\
 &=  ( \frac{d}{\Delta} (tx^2 y +2y^3) +  \frac{b}{\Delta} (ty^2 x +2x^3)  )\partial_{x'}  - ( \frac{c}{\Delta} (tx^2 y +2y^3) +  \frac{a}{\Delta} (ty^2 x +2x^3)  )\partial_{y'} .
 \end{align*} 
 Since the line $H_1$ is invariant under $ \phi_*$, we may assume the equality
 \[
	\phi_*(h_1(t)) =  \lambda \cdot h_1(s)  
 \]
 holds for some constant $\lambda $.
 This is equivalent to 
 \[
	( \frac{d}{\Delta} (tx^2 y +2y^3) +  \frac{b}{\Delta} (ty^2 x +2x^3)  ) =\lambda( s x'^2 y' +2y'^3),
 \]
 and 
 \[
	( \frac{c}{\Delta} (tx^2 y +2y^3) +  \frac{a}{\Delta} (ty^2 x +2x^3)  ) = \lambda( s y'^2 x' +2x'^3).
 \]
 Therefore, we obtain:
 \begin{align*}
	 2 \lambda (x'^4 + y'^4 + s x'^2 y'^2) 
	= & \lambda( s x'^2 y' +2y'^3) y' + \lambda( s y'^2 x' +2x'^3) x'\\
	 = &(tx^2 y +2y^3) y + (ty^2 x +2x^3) x \\
	 =& 2 (x^4 + y^4 + tx^2 y^2).
 \end{align*}
 It follows that $ \phi $ comes from $ \Grp^{\infty}(\tilde{E}_7)(t,s) $. Therefore, the case $k=1$ of the assertion (1) in Theorem \ref{thmB} is proved.
 \end{proof}
 
\bibliographystyle{IEEEtran}
 \bibliography{paper}

\end{document}